\documentclass[12pt]{amsart}

\usepackage{amssymb,latexsym}

\usepackage{graphicx,epsfig}
\usepackage[dvips]{color}

\def\cal{\mathcal}

\def\Bbb{\mathbb}

\def\ord{\text{\rm ord\,}}

\def\pad{\phi^a}
\def\ra{\rangle}
\def\laa{\langle}

\def \supp {\text{\rm supp\,}}

\def\N{{\cal N}}
\def\S{{\cal S}}
\def\T{{\cal T}}
\def\bC{{\Bbb C}}

\def\NN{{\Bbb N}}
\def\bN{{\Bbb N}}
\def\bR{{\Bbb R}}
\def\RR{{\Bbb R}}
\def\bZ{{\Bbb Z}}

\def\vp{{\varphi}}
\def\al{{\alpha}}
\def\be{{\beta}}
\def\ga{{\gamma}}
\def\la{{\lambda}}
\def\om{{\omega}}
\def\x{(x_1,x_2)}
\def\pa{{\partial}}

\def\ve{{\varepsilon}}
\def\si{{\sigma}}

\def\de{{\delta}}
\def\dee{{\delta^e}}

\def\Om{{\Omega}}
\def\ka{{\kappa}}

\def\pr{\text{\rm pr\,}}

\def\bpm{\begin{pmatrix}}
\def\epm{\end{pmatrix}}

\def\noi{\noindent}
\def\bee{\begin{enumerate}}
\def\ee{\end{enumerate}}

\def\qed{\smallskip\hfill Q.E.D.\medskip}

\newtheorem{thm}{Theorem}[section]
\newtheorem{prop}[thm]{Proposition}
\newtheorem{cor}[thm]{Corollary}
\newtheorem{lemma}[thm]{Lemma}

\newtheorem{remarks}[thm]{Remarks}

\begin{document}

\title[Uniform estimates and a sharp restriction theorem]
{Uniform estimates for the Fourier transform of surface carried measures in  $\bR^3$ and an application to Fourier restriction. }

\author[I. A. Ikromov]{Isroil A.Ikromov}
\address{Department of Mathematics, Samarkand State University,
University Boulevard 15, 140104, Samarkand, Uzbekistan}
\email{{\tt ikromov1@rambler.ru}}

\author[D. M\"uller]{Detlef M\"uller}
\address{Mathematisches Seminar, C.A.-Universit\"at Kiel,
Ludewig-Meyn-Stra\ss{}e 4, D-24098 Kiel, Germany}
\email{{\tt mueller@math.uni-kiel.de}}
\urladdr{{http://analysis.math.uni-kiel.de/mueller/}}

\thanks{We acknowledge the support for this work by the Deutsche Forschungsgemeinschaft.}

\begin{abstract}
Let $S$ be a  hypersurface in $\RR^3$ which is the graph of a smooth,  finite type function $\phi,$  and let $\mu=\rho\, d\si$ be a surface carried measure on $S,$ where $d\si$ denotes the surface element on $S$ and $\rho$ a smooth density with suffiently small support. We derive uniform estimates for the Fourier transform $\hat \mu$ of $\mu,$  which are sharp except for the case where the principal face of the Newton polyhedron of  $\phi,$ when expressed in adapted coordinates, is unbounded. As an application, we prove  a sharp $L^p$-$L^2$ Fourier restriction theorem for $S$ in the case where the original coordinates are  adapted to $\phi.$ This improves on earlier joint work  with  M. Kempe.
\end{abstract}

\maketitle


\tableofcontents

\thispagestyle{empty}

\section{Introduction}\label{intro}


The goal of this article is to improve on two results from our previous article \cite{IKM-max} concerning  uniform estimates for two-dimensional oscillatory integrals with smooth, finite type phase functions, and   $L^p$-$L^2$ Fourier restriction  for smooth, finite type hypersurfaces $S$ in $\RR^3$ which are locally the graph of a function $\phi$  in adapted coordinates.

More precisely, we shall  identify in Theorems \ref{s1.1} and \ref{limit} exactly when the logarithmic factor  in the estimate (1.6) for the Fourier transform of a surface carried measure of $S$ in Theorem 1.9 of \cite{IKM-max} Êwill be present, with the exception of the case  where the principal face of the Newton polyhedron of  $\phi,$ when expressed in adapted coordinates, is unbounded and $\phi$ is non-analytic. Examples by A.  Iosevich and E. Sawyer show that a different behavior can indeed occur in the latter case. Moreover,  we shall show in Theorem \ref{restrict} that the restriction theorem Corollary 1.10 of that article can be improved  as follows:

\smallskip
Assume that  $S$ is represented near a point $x^0$ as the graph of a function $\phi\x,$ and that, after translation of coordinates, $x^0=0.$
 Then, in the case where the coordinates $\x$ are adapted to $\phi$ after applying a linear change of coordinates,  the restriction estimate holds true also at the endpoint $p'= 2h(\phi)+2,$ where $h(\phi)$ denotes the height of $\phi$ in the sense of Varchenko.

 \smallskip
If the coordinates $\x$ are not adapted to $\phi,$  then we will show in a sequel  to this article that the restriction estimate can be extended to an even  wider range of $p$'s.
 \medskip

 We shall build on the results and techniques developed in \cite{IM-ada} and \cite{IKM-max}, which will be our main references, also for cross-references to earlier and related work. Let us first recall some basic notions from \cite{IM-ada}, which essentially  go back to A.\,N.~ Varchenko \cite{Va}.
\smallskip

Let $\phi$ be a smooth real-valued  function defined on a neighborhood of the origin in $\bR^2$ with $\phi(0,0)=0,\, \nabla \phi(0,0)=0,$ and consider the associated Taylor series
$$\phi(x_1,x_2)\sim\sum_{j,k=0}^\infty c_{jk} x_1^j x_2^k$$
of $\phi$ centered at  the origin.
The set
$$\T(\phi):=\{(j,k)\in\bN^2: c_{jk}=\frac 1{j!k!}\partial_{ 1}^j\partial_{ 2}^k \phi(0,0)\ne 0\}
$$
will be called the {\it Taylor support } of $\phi$ at $(0,0).$  We shall always assume that
$$\T(\phi)\ne \emptyset,$$
i.e., that the function $\phi$ is of finite type at the origin.
The {\it Newton polyhedron} $\N(\phi)$ of $\phi$ at the origin is defined to be the convex hull of the union of all the quadrants $(j,k)+\bR^2_+$ in $\bR^2,$ with $(j,k)\in\T(\phi).$  The associated {\it Newton diagram}  $\N_d(\phi)$ in the sense of Varchenko \cite{Va}  is the union of all compact faces  of the Newton polyhedron; here, by a {\it face,} we shall  mean an edge or a vertex.

We shall use coordinates $(t_1,t_2)$ for points in the plane containing the Newton polyhedron, in order to distinguish this plane from the $(x_1,x_2)$ - plane.

The {\it Newton distance}, or shorter {\it distance} $d=d(\phi)$  between the Newton
polyhedron and the origin in the sense of Varchenko is given by the coordinate $d$ of the point $(d,d)$ at which the bi-sectrix   $t_1=t_2$ intersects the boundary of the Newton polyhedron.

The {\it principal face} $\pi(\phi)$  of the Newton polyhedron of $\phi$  is the face of minimal dimension  containing the point $(d,d)$. Deviating from the notation in \cite{Va}, we shall call the series
$$\phi_\pr(x_1,x_2):=\sum_{(j,k)\in \pi(\phi)}c_{jk} x_1^j x_2^k
$$
the {\it principal part} of $\phi.$ In case that $\pi(\phi)$ is compact,  $\phi_\pr$ is a mixed homogeneous polynomial; otherwise, we shall  consider $\phi_\pr$ as a formal power series.

Note that the distance between the Newton polyhedron and
the origin depends on the chosen local coordinate system in which $\phi$ is expressed.  By a  {\it local  coordinate system at the origin} we shall mean a smooth   coordinate system defined near the origin which preserves $0.$ The {\it height } of the  smooth function $\phi$ is defined by
$$h(\phi):=\sup\{d_x\},$$
 where the
supremum  is taken over all local  coordinate systems $x=(x_1,x_2)$ at the origin, and where $d_x$
is the distance between the Newton polyhedron and the origin in the
coordinates  $x$.

A given   coordinate system $x$ is said to be
 {\it adapted} to $\phi$ if $h(\phi)=d_x.$
In \cite{IM-ada} we proved that one  can always find an adapted local  coordinate system in two dimensions, thus  generalizing  the fundamental work by Varchenko  \cite{Va} who worked in the   setting of  real-analytic functions $\phi$ (see also \cite{PSS} for another proof in the analytic case).

 \medskip
 Following  \cite{Va} (with as slight modification), we next  define what we like to call {\it Varchenko's   exponent}  $\nu(\phi)\in\{0,1\}$ as follows (this number had been identified by  Varchenko in \cite{Va} as what Karpushkin calls  the ''multiplicity of the oscillation of $\phi$ at $(0,0)$'' in \cite{karpushkin}) :
 \smallskip

 If there exists an adapted local coordinate system $y$ near the origin such that the principal face $\pi(\pad)$ of $\phi,$ when expressed by the function $\pad$ in the new coordinates (i.e. $\phi(x)=\pad(y)$), is a vertex, and if $h(\phi) \ge 2,$ then we put $\nu(\phi):=1;$ otherwise, we put $\nu(\phi):=0.$

 \smallskip
 As  has been shown by  Varchenko in \cite{Va},  the number   $\nu(\phi)$ arises as the
 exponent of a logarithmic factor  in the principal part of the asymptotic expansion of two-dimensional oscillatory integrals with real analytic phase functions $\phi.$
\smallskip

Analogously, we can prove the following uniform estimate for two-dimensional oscillatory integrals with smooth, finite type phase functions $\phi,$ which improves on  Theorem 11.1 in \cite{IKM-max}.

\begin{thm}\label{s1.1}
Let $\phi$ be a smooth, real-valued phase function of finite type, defined near the origin, as before, and let $h:=h(\phi), \nu:=\nu(\phi).$ 
Then there exist a neighborhood $\Om\subset \RR^2$ of the origin and a constant $C$ such
that for every  $\eta\in C_0^\infty(\Om)$ the following estimate holds true for every $\xi\in\RR^3:$
\begin{equation}\label{1.1}
\Big|\int_{\RR^2}e^{i(\xi_3\phi\x+\xi_1x_1+\xi_2x_2)}\eta(x)\, dx\Big|
\le C\,\|\eta\|_{C^3(\RR^2)}\,(\log(2+|\xi|))^{\nu}(1+|\xi|)^{-1/h}.
\end{equation}
\end{thm}

\smallskip
\begin{remarks}\label{rems}
 {\rm
(a) For  some special classes of hypersurfaces, related results have been derived by L.  Erd\"os and M. Salmhofer in  \cite{erdšs}, which, however, are  not necessarily uniform in all directions. 
For estimates  with $\xi_1=\xi_2=0,$ we refer to 
the recent work of M. Greenblatt \cite{greenblatt}.

\smallskip
(b)  For real analytic phase functions $\phi,$ if we restrict ourselves to the direction where $\xi_1=\xi_2=0,$ then the asymptotic expansion of the corresponding oscillatory integrals in \cite{Va} shows that the estimate \eqref{1.1} is essentially sharp as an estimate in terms of  $|\xi|.$

\smallskip
(c)  For real analytic phase functions, our result is covered by  Karpushkin's work \cite{karpushkin}, who proved the following:

If
$\phi$ is a real analytic function defined near the origin with
$\phi(0,0)=0,\, \nabla\phi(0,0)=0,$ and if $r$ is a real  analytic function
with sufficiently small norm (in the space of real  analytic functions)
then
$$
\left|\int_{\bR^2} e^{i \la(\phi(x)+r(x))}\eta(x)\, dx\right| \le
C\|\eta\|_{C^3}\frac{(\log(2+|\la|))^{\nu}}{(2+|\la|)^{1/h}} ,\quad \la\in\RR,
$$
provided the amplitude  $\eta$ is supported in a sufficiently
small neighborhood of the origin. Moreover, the constant $C$ then does 
not depend on the function $r$.

\smallskip
(d) If $h(\phi)<2,$   results analogous to Karpushkin's have been
obtained by  J.\,J. Duistermaat \cite{duistermaat} in the smooth setting. In this case one always has  $\nu(\phi)=0$.

\smallskip
(e)  If $h(\phi)=2,$ and if the principal part of
$\phi,$  when expressed in an adapted coordinate system, has a critical point of finite multiplicity at the origin (so that it is isolated), then an analogue to Karpushkin's estimate has been established by  Colin de Verdi\`ere   \cite{CDV} in the smooth setting. Notice  that if the
principal part of $\phi$ has an isolated critical point at the origin,
then the coordinate system is adapted to $\phi$ and
$\nu(\phi)=0.$
}
\end{remarks}

\medskip

The next result,  which improves on corresponding results by M. Greenblatt, shows in particular that, in  most cases, the uniform estimates from Theorem \ref{s1.1} are sharp if $(\xi_1,\xi_2)=(0,0).$

\begin{thm}\label{limit}
Let us put 
$$
J_\pm(\la):=\int_{\RR^2}e^{\pm i\la \phi\x}\eta(x)\, dx,\qquad \la>0,
$$
with $\phi$ and $\eta$ as in Theorem \ref{s1.1}. If the principal face $\pi(\pad)$ of $\phi,$ when given in adapted coordinates,  is a compact set (i.e., a compact edge or a vertex), then there exists a neighborhood  $\Om$ of the origin such that for every $\eta$ supported in $\Om$ 
the following limits 
\begin{equation}\label{limitpm}
\lim_{\la\to +\infty} \frac{\la^{1/h}}{(\log{\la})^{\nu}}J_\pm(\la)=c_\pm \,\eta(0)
\end{equation}
exist, where the constants  $c_\pm$ are non-zero and depend on the phase function $\phi$ only.
 \end{thm}

\begin{remarks}\label{ios}
{\rm
(a) The proof of Theorem \ref{limit} reveals the following additional facts:

 If $\nu(\phi)=0$ in the theorem, then the principal face $\pi(\pad)$ is a compact edge, and the constants $c_\pm$ are completely determined by the principal part $\pad_\pr$ of $\pad.$ And, if $\nu(\phi)=1,$  and if we work in super-adapted coordinates in the sense of Greenblatt (as explained in Lemma \ref{supera}),  so that in particular  $\pi(\pad)$ consists of  the vertex $(h,h),$ then the constants $c_\pm$ are completely determined by the principal part $\pad_\pr$ of $\pad$  and the slopes of those compact edges of $\N(\pad)$ which  contain this vertex.
\smallskip

(b) An analogous result for real analytic phase functions $\phi$  has been proven 
by M. Greenblatt (Theorem 1.2 in \cite{greenblatt}). For non-analytic, but  smooth and finite type  $\phi,$ 
the following weaker result had been obtained in Theorem 1.6b of the same article:
$$
\limsup_{\la\to +\infty} \Big|\frac{\la^{1/h}}{(\log{\la})^{\nu}}J_\pm(\la)\Big|>0.
$$

\smallskip

(c) If the principal face $\pi(\pad)$ is unbounded, then the estimate in 
Theorem \ref{s1.1}  may fail to be sharp, if $\phi$ is non-analytic, as the following class of examples by A.  Iosevich and E. Sawyer \cite{iosevich-sawyer}
shows: If  $$
\phi(x_1,x_2):=x_2^2+e^{-1/|x_1|^\alpha},
$$
with $\al>0,$
then 
$$
|J_\pm(\la) |\asymp \frac1{\lambda^{1/2}\log{\lambda}^{1/\alpha}}\quad \mbox{as}\quad \la\to+\infty,
$$
whereas $\nu(\phi)=0.$ These examples also indicate that a precise determination of the asymptotic behavior of $J_\pm(\la)$ may be difficult when the principal face is non-compact.

\smallskip
(d) For  real-analytic phase functions  depending on more than two variables and  satisfying  an appropriate non-degeneracy condition, the explicit form of the 
principal part of the asymptotic expansion of the  corresponding oscillatory
integrals has been obtained by J. Denef, J. Nicaise and  P. Sargos \cite{denef-etal}.  

}
\end{remarks}

\bigskip
The existence of an adapted coordinate system in which the principal face is a vertex is a priori not so easily verified, but there exists an equivalent, more accessible  condition.
In order to describe this, we first recall that if the principal face of the Newton polyhedron $\N(\phi)$ is a compact edge, then it lies on a unique line $\ka_1t_1+\ka_2t_2=1,$ with $\ka_1,\ka_2>0.$ By permuting the  coordinates $x_1$ and $x_2,$ if necessary, we shall always assume that $\ka_1\le\ka_2.$ We shall call  this weight $\ka=(\ka_1,\ka_2)$ the {\it principal weight} associated to $\phi,$ and  denote it also by $\ka^\pr.$  It induces dilations $\de_r\x:=(r^{\ka_1}x_1,r^{\ka_2} x_2),\ r>0,$ on $\RR^2,$ so that the principal part $\phi_\pr$ of $\phi$ is $\ka$- homogeneous of degree one with respect to these dilations, i.e.,  $\phi_\pr(\de_r\x)=r\phi_\pr\x$ for every $r>0,$ and 
\begin{equation}\label{dnew}
d=\frac 1{\ka_1^\pr+\ka_2^\pr}=\frac1{|\ka^\pr|}.
\end{equation}

Denote  by

$$m(\phi_\pr):=\ord_{S^1} \phi_\pr$$

the maximal order of vanishing of $\phi_\pr$ along the unit circle $S^1$ centered at the
origin.

We also recall from  \cite{IM-ada} that the {\it homogeneous distance} of a $\ka$-homogeneous polynomial $P$ (such as $P=\phi_\pr$) is given by $
d_h(P):= 1/{(\ka_1+\ka_2)}=1/|\ka|,$
and that
\begin{equation}\label{heightp}
h(P)=\max\{m(P), d_h(P)\}.
\end{equation}

 According to  \cite{IM-ada}, Corollary 4.3 and  Corollary 2.3, {\it the coordinates $x$ are adapted to $\phi$ if and only if one of the following conditions is satisfied:
\medskip

\bee
\item[(a)]  The principal face  $\pi(\phi)$ of the Newton polyhedron  is a compact edge, and $m(\phi_\pr)\le d(\phi).$
\item[(b)] $\pi(\phi)$  is a vertex.
\item[(c)] $\pi(\phi)$ is an unbounded edge.
\ee}
\medskip

We like to mention that in case (a) we have $h(\phi)=h(\phi_\pr)=d_h(\phi_\pr).$  Notice also that (a) applies whenever $\pi(\phi)$ is a compact edge  and $\ka_2/\ka_1\notin\NN;$   in this case we even have  $m(\phi_\pr)< d(\phi)$ (cf. \cite{IM-ada}, Corollary 2.3).




\begin{lemma}\label{vertex}
The following conditions on $\phi$ are equivalent:
\bee
\item[(a)] There exists an adapted local coordinate system $y$ at  the origin such that the principal face $\pi(\pad)$  is a vertex.

\item[(b)]  If $y$ is any adapted local coordinate system at the origin, then either $\pi(\pad)$ is a vertex, or a compact edge and $m(\pad_\pr)=d(\pad).$
 \ee
\end{lemma}

Consider for example  the function  $\phi\x:=(x_2-2x_1^2)^2(x_2-x_1^2).$ Then $\phi=\phi_\pr,$ $\pi(\phi)$ is a compact edge and $m(\phi_\pr)=2=d(\phi),$ so that  case (b) above applies and the coordinates $x$ are adapted to $\phi.$ Moreover, $\nu(\phi)=1.$  If we introduce new coordinates $y$ given by $y_1:=x_1,\  y_2:=x_2-2x_1^2,$ then $\phi(x)=\tilde\phi(y),$ where   $\tilde\phi(y)=y_2^2(y_2+y_1^2).$ The principal face of $\N(\tilde\phi)$ is the vertex $(2,2),$ so that also the coordinates $y$ are adapted.
\medskip

In the case where  the coordinates are not adapted to $\phi,$ we see that the principal face $\pi(\phi)$ is a  compact edge such that
\begin{equation}\label{m1}
m_1:=\ka_2/\ka_1\in\NN.
\end{equation}

Then, by Theorem 5.1 in \cite{IM-ada}, there exists a smooth real-valued function $\psi$ of the form
\begin{equation}\label{prjet}
\psi(x_1)=b_1x_1^{m_1}+O(x_1^{m_1+1}),
\end{equation}
 with $b_1\ne 0,$ defined on a neighborhood of the origin  such that an adapted  coordinate system $(y_1,y_2)$ for $\phi$ is given locally near the origin by means of the (in general non-linear) shear
$$
y_1:= x_1, \ y_2:= x_2-\psi(x_1).
$$
  In these coordinates, $\phi$ is given by
\begin{equation}\label{phia1}
 \phi^a(y):=\phi(y_1,y_2+\psi(y_1)).
\end{equation}




\bigskip

As an  immediate consequence of Theorem \ref{s1.1} we obtain  uniform  estimates for the Fourier transform
$$\widehat{\rho d\si}(\xi)=\int_S e^{-i\xi\cdot x}\rho(x)\, d\si(x),\quad \xi\in\RR^3,
$$
 of surface carried  measures on smooth, finite type hypersurfaces $S$  in $\RR^3.$ Here, $d\si$ denotes the Riemannian volume element on $S.$

 If a point $x^0$ on such a hypersurface $S$ is given, which we may assume to be the origin after a translation of coordinates, and if we represent $S$ locally near $x^0=(0,0)$ as the  graph $x_3=\phi\x$ of smooth, finite type  function $\phi$ with $\phi(0,0)=0,\nabla\phi(0,0)=0$ as before, then we define the {\it height} of $S$ at $x^0$ by $h(x^0,S):=h(\phi).$ This notion is invariant under affine linear coordinate changes  of the ambient space, as has been shown in \cite{IKM-max}. Similarly, we define  $\nu(x^0,S):=\nu(\phi).$ Denote by $d\si$ the surface element of $S.$  Then we have the following improvement of Theorem 1.9 in  \cite{IKM-max}:

\begin{cor}\label{s1.3}
Let $S$ be  a smooth hypersurface of finite type in $\RR^3$ and let $x^0$ be a fixed point on $S.$ Then there exists a neighborhood $U\subset S $ of the point $x^0$ such
that for every  $\rho\in C_0^\infty(U)$ the following estimate holds true:
$$
|\widehat{\rho d\si}(\xi)|\le C\,\|\rho\|_{C^3(S)}\,(\log(2+|\xi|))^{\nu(x^0,S)}(1+|\xi|)^{-1/h(x^0,S)}\  \mbox{ for every } \xi \in\RR^3.
$$
\end{cor}

\bigskip
Our second result concerns  Fourier restriction to  $S.$ We shall prove that the $L^p$-$L^2$ Fourier  restriction theorem of Corollary 1.10 in \cite{IKM-max} also holds true at the endpoint, if $S$ is locally given as the graph of a function $\phi$ which is given in adapted coordinates:

\begin{thm}
\color{black}\label{restrict}
Let $S$ be  a smooth hypersurface of finite type in $\RR^3,$ and let $x^0$ be a fixed point on $S.$ Assume that, possibly after a translation of coordinates, $x^0=0,$ and that  $S$ is locally near $x^0$ given as  the  graph $x_3=\phi\x$ of a smooth, finite type  function $\phi$ with $\phi(0,0)=0,\nabla\phi(0,0)=0$ as before.

We also assume that, after applying a suitable linear change of coordinates, the coordinates $\x$ are adapted to $\phi,$ so that $d=h,$ where $d=d(\phi)$ denotes the Newton distance of $\phi$ and $h=h=h(x^0,S)=h(\phi)$ its height.
We then define the critical exponent $p_c$ by
\begin{equation}\label{pcritical}
p'_c:=2h +2,
\end{equation}
where $p'$ denotes the exponent conjugate to $p,$ i.e., $1/p+1/p'=1.$

Then there exists a neighborhood $U\subset S $ of the point $x^0$ such
that for every  non-negative density $\rho\in C_0^\infty(U)$ the  Fourier restriction estimate
\begin{equation}\label{rest1}
\Big(\int_S |\widehat f|^2\, \rho \,d\si\Big)^{1/2}\le C_p\|f\|_{L^p(\RR^3)},\qquad  f\in\S(\RR^3),
\end{equation}
holds true for every $p$ such that
\begin{equation}\label{rest2}
1\le p\le p_c.
\end{equation}

Moreover, if $\rho(x^0)\ne 0,$ then the condition \eqref{rest2} on $p$ is also necessary for the validity of  \eqref{rest1}.

\end{thm}

The second statement had already been proven in \cite{IKM-max}, Section 12,  so that we shall only have  to prove the first statement, for the endpoint $p=p_c.$

\begin{remarks}\label{s1.5}
{\rm
(a) The case where  the coordinates $\x$ are not  adapted to $\phi$  will be treated in a subsequent article. It has turned out that in this case the restriction estimate \eqref{rest1} is valid in a wider range  of $p's,$  with a critical exponent which is strictly bigger than $p_c$ and which can  be determined explicitly by means of Varchenko's algorithm (cf. \cite{IM-ada}) for the construction of adapted coordinates.

\smallskip
(b) 
If the surface $S$  is of finite line type and  convex,  and if  the restriction property \eqref{rest1} holds true also in the endpoint $p_c=(2h+2)/(2h+1),$    then it has been shown  by A.~Iosevich in  \cite{iosevich} that  necessarily the Fourier transform of $\rho \, d\si$  must   decay of order $O(|\xi|^{-1/h})$ as $|\xi|\to+\infty$  (it can easily  be shown by means of  Schulz' \cite{schulz}  decomposition of  convex smooth   functions of finite line  type), i.e., $\nu(x^0,S)=0.$ Conversely, the decay rate  $O(|\xi|^{-1/h})$ immediately implies the restriction estimate  \eqref{rest1} also for the endpoint  $p=p_c;$ this is an immediate consequence of A.~Greenleaf's work in \cite{greenleaf}.

However, if $\nu(x^0,S)=1, $ which can only happen in the non-convex case, the logarithmic factor in \eqref{1.1} is necessary, so that one  cannot apply Greenleaf's result directly.

}
\end{remarks}

Restriction theorems for the Fourier transform have a long history by now, starting with the seminal work by E.M.~Stein, and P. Tomas, for the case of the Euclidean sphere (see, e.g., \cite{stein-book}).
Some restriction estimates for  analytic hypersurfaces in $\RR^3$ have been obtained  by  A. ~Magyar \cite{magyar}, whose results were sharp for particular classes of hypersurfaces given as graphs of functions in adapted coordinates, with the exception of the  endpoint.

\setcounter{equation}{0}
\section{Uniform estimates for oscillatory integrals with finite type phase functions of two variables}\label{uniest}

In this section we shall give  a proof of Theorem \ref{s1.1}. We shall closely  follow the proof of Theorem 11.1 in  \cite{IKM-max},
which did already provide  the uniform estimates in Theorem \ref{s1.1}, except for a logarithmic factor which is not really needed in many cases, as we shall see.

 The reader is strongly recommended to have  \cite{IKM-max} at hand when reading this article, since we shall make use of the notation and many results from \cite{IKM-max} without repeating all of them here.
  \medskip

  By decomposing $\RR^2$ into its four quadrants, we may reduce ourselves to the estimation of oscillatory
integrals of the form
$$
J(\xi):=\int_{(\RR_+)^2}e^{i(\xi_3\phi\x+\xi_1x_1+\xi_2x_2)}\eta\x\, dx.
$$
Notice also that we may assume in the sequel that
\begin{equation}\label{2.1}
|\xi_1|+|\xi_2|\le \de |\xi_3|,\quad\mbox{hence} \ |\xi|\sim |\xi_3|,
\end{equation}
where $0<\de\ll1$ is a  sufficiently small constant to be chosen later, since for $|\xi_1|+|\xi_2|> \de |\xi_3|$ the estimate \eqref{1.1} follows by an integration by parts, if $\Om$ is chosen small enough.
Of course, we may in addition always assume that $|\xi|\ge 2.$

If $\chi$ is any integrable function defined on $\Om,$ we shall put
$$J^\chi(\xi):=\int_{(\RR_+)^2}e^{i(\xi_3\phi\x+\xi_1x_1+\xi_2x_2)}\eta\x\chi(x)\, dx.
$$

\medskip
The case where $h(\phi)<2$ is contained in Duistermaat's work \cite{duistermaat}  (notice that Duistermaat proves estimates of the form \eqref{1.1}  without the presence of a logarithmic factor $\log(2+|\xi|),$ even for a wider  class of  phase functions),  so let us assume from now on that
$$h:=h(\phi)\ge 2.
$$
Moreover, if $h=2,$ then we shall make use of the following special property:

\begin{lemma}\label{s2.1}
If $h(\phi)=2,$ then, after applying a suitable  linear change of coordinates, we may assume that one of the following conditions are satisfied:
\bee
\item[(i)] The coordinates are adapted to $\phi.$
\item[(ii)] The coordinates are not adapted to $\phi,$ but $h(\phi_\pr)=h(\phi).$
In this  case, we have $\nu(\phi)=0$ and $m(\phi_\pr)=2.$
\ee
\end{lemma}
Note that in general we only have $h(\phi_\pr)\ge h(\phi),$ and  the inequality may be strict.

\proof
Let us assume that the coordinates $x$  are not adapted to $\phi.$ Then the principal face $\pi(\phi)$ is a compact edge and $m(\phi_\pr)>d(\phi)=d_x.$ In particular, the principal part $\phi_\pr$ of $\phi$ is a polynomial which is $\ka$-homogeneous of degree $1,$ where we may assume that $0<\ka_1\le \ka_2,$ so that  $m:=m_1=\ka_2/\ka_1\ge 1$ is an integer.  According to \cite{IM-ada}, $\phi_\pr$ can be written as
$$
\phi_\pr(x_1,x_2)=cx_1^{\al}x_2^{\beta}\prod_{l}(x_2-c_lx_1^{m})^{n_l},
$$
where the $c_l$'s are the non-trivial distinct complex roots of the polynomial $t\mapsto \phi_\pr(1, t)$ and the $n_l$'s  are their multiplicities. Moreover,  there exists an $l_0$ such that $m(\phi_\pr)=n_{l_0}$ and such that $c_{l_0}$ is real. Notice also that $\al\le 1,\be\le 1,$ since otherwise the coordinates were adapted.

Assume first that $\ka_1=\ka_2.$ Then $m=1,$ and applying the first step in Varchenko's algorithm (see  \cite{IM-ada}, or Subsection 2.5 in \cite{IKM-max}), we see that we can transform $\phi$ into $\tilde \phi$ by means of the linear change of variables $y_1=x_1, y_2=x_2-c_{l_0}x_1$ such that either the coordinates $y$ are adapted to $\tilde\phi,$ hence $\tilde\phi=\pad,$ or they are not, but then $\tilde\ka_1<\tilde\ka_2$ (where $\tilde\phi_\pr$ is assumed to be $\tilde\ka$-homogeneous of degree $1$).

After applying a suitable linear change of coordinates, we are thus reduced to the situation where $\ka_1<\ka_2,$ hence $m\ge 2.$ Let us denote by $(A_0,B_0)$ and $(A_1,B_1)$ the two vertices of $\pi(\phi),$  and assume that $A_0< A_1.$ Recall from \cite{IM-ada}, displays (3.2) and (3.3), that
$$
A_0=\al,\, B_0=\beta+N,\  \quad
A_1=\al+mN,\quad B_1=\beta,
$$
and that
\begin{equation}\label{dist}
d_x=\frac {\al+m(\beta+N)}{1+m},
\end{equation}
with $N:=\sum_{l}n_l.$
Recall also that the point $(A_0,B_0),$ with $A_0<B_0,$ will be a vertex of all  the Newton diagrams that arise when running Varchenko's algorithm on $\phi,$ so that we must have $\be+N=B_0\ge 2,$ since $h(\phi)=2.$
Then \eqref{dist} implies that $d_x\ge \frac{2m}{1+m}>1,$ so that $n_{l_0}=m(\phi)\ge 2.$

Since $d_x\le h(\phi)=2,$ \eqref{dist} implies that $\be+N\le 2+\frac 2m\le 3.$  But, if we had $\be+N= 3,$ then  the conditions $d_x\le 2$  and $m\ge 2$ would  imply  $\al=0, m=2,$ hence $d_x=2,$ and so the coordinates $x$ would be adapted, contradicting our assumption.
\smallskip

Therefore, we must have  $\be+N= 2.$ Then $\be=0, N=n_{l_0}=2$ and $\al<2,$  and thus the change of coordinates
$$
y_1:=x_1,\ y_2:=x_2-c_{l_0}x_1^m
$$
transforms the principal part $\phi_\pr$ into $\widetilde{\phi_\pr}(y)=cy_1^\al y_2^2.$ This implies $h(\phi_\pr)=2=h(\phi).$ Notice also that in this case the principal face of the Newton polyhedron of $\phi,$ when expressed in adapted coordinates, must be the unbounded half-line with left endpoint $(\al,2),$ so that $\nu(\phi)=0.$

\qed

We recall the following lemma, which is a (not quite straight-forward) consequence of  van der Corput's  lemma  and whose formulation goes back to  J.\,E. Bj\"ork  (see \cite{domar}) and  G.\,I.~Arhipov \cite{arhipov}.

\begin{lemma}\label{s2.2}
Assume that $f$ is a smooth real valued function defined on an interval $I\subset \RR$ which is of polynomial type  $m\ge 2\ (m\in\NN)$, i.e., there are positive constants $c_1,c_2>0$ such that
$$
c_1\le\sum^m_{j=2}|f^{(j)}(s)|\le c_2\quad\mbox{for every}\  s\in I.
$$
Then for $\la\in\RR,$
$$
\Big|\int_{I}e^{i\la f(s)} g(s)\, ds\Big| \le C\| g\|_{C^1(I)} (1+|\la|)^{-1/m},
$$
where the constant $C$ depends only on the constants $c_1$ and $c_2.$
\end{lemma}

\subsection{The case where the coordinates are  adapted to $\phi,$  or where $h=2$}\label{adaptedc}
\medskip

We  shall  begin with  the easiest case where either the coordinates $x$ are adapted to $\phi,$ or $h=2$ and  condition (ii) in Lemma \ref{s2.1} is satisfied.

Recall from \cite{IM-ada} that if $\ka=(\ka_1,\ka_2)$ is any weight with $0<\ka_1\le \ka_2$ such that  the line  $L_\ka:=\{(t_1,t_2)\in\bR^2:\ka_1 t_1+\ka_2t_2=1\}$ is a supporting  line to the Newton polyhedron $\N(\phi) $ of $\phi,$ then  the {\it $\ka$-principal part} of $\phi$
$$
\phi_\ka(x_1,x_2):=\sum_{(j,k)\in L_\ka} c_{jk} x_1^j x_2^k
$$
 is a non-trivial polynomial  which is $\ka$-homogeneous of degree $1.$  By definition, we then have
$$
\phi(x_1,x_2)=\phi_\ka(x_1,x_2) +\ \mbox{terms of higher $\ka$-degree}
$$
(see \cite{IM-ada} for the precise meaning of this notion).

We claim that we can choose a weight $\ka$ with  $0<\ka_1\le\ka_2<1$ such that  $L_\ka$ is a supporting  line to the Newton polyhedron of $\phi$ and 
$$
\frac1{|\ka|}=d_h(\phi_\ka)\le h(\phi_\ka)=h.
$$
Indeed, in  case that the coordinates are adapted to $\phi,$ this has been shown in \cite{IKM-max}, Lemma 2.4.   And, if the coordinates are not adapted to $\phi$ but $h(\phi_\pr)=h(\phi),$ then the principal face is a compact edge, and we can choose for $\ka$ the principal weight, so that $\phi_\ka=\phi_\pr.$ Notice that we have $\ka_2<1,$ since $\nabla\phi(0,0)=0.$

Let us denote by $\de_r$  the dilation by the factor $r>0$ associated to the weight $\ka,$ i.e., $\de_r\x=(r^{\ka_1}x_1,r^{\ka_2} x_2).$

 In analogy with the proof of Theorem 11.1 in  \cite{IKM-max}  we fix a suitable smooth cut-off function $\chi$ on $\RR^2$ supported in an annulus $D$ such that the functions $\chi_k:=\chi\circ \de_{2^k}$ form a partition of unity,  and then decompose 
 $$J(\xi)=\sum_{k=k_0}^\infty J_k(\xi),$$
  where
\begin{eqnarray*}
J_k(\xi)&:=&\int_{(\RR_+)^2}e^{i(\xi_3\phi(x)+\xi_1x_1+\xi_2x_2)}\eta(x)\chi_k(x)\, dx\\
&=& 2^{-k|\ka|} \int_{(\RR_+)^2}e^{i\Big(2^{-k}\xi_3\phi^k(x)+2^{-k\ka_1}\xi_1x_1+2^{-k\ka_2}\xi_2x_2\Big)}\eta(\de_{2^{-k}}(x))\chi(x)\, dx,
\end{eqnarray*}
with $\phi^k(x):=2^k\phi(\de_{2^{-k}}x)=\phi_\ka (x)+ \mbox{error term}.$

\smallskip

 We claim that given any point $x^0\in D,$ we can
find a unit vector $e\in\RR^2$ and some $j\in
 \NN$  with $2\le j\le h(\phi_\ka)=h$ such that $\pa_e^{j}\phi_\ka(x^0)\ne 0.$
 
 Indeed, if the coordinates are adapted to $\phi,$ then this has been shown in  Section 7 of \cite{IKM-max}, and if they are not adapted to $\phi,$ then the same is true whenever $x^0$ does not lie on the principal root of $\phi_\ka,$ as shown in Section 8 of \cite{IKM-max}. However, if $x^0$ does lie an the principal root of $\phi_\ka,$ then according to Lemma \ref{2.1} we may choose $j=2.$

\smallskip

For $k\ge k_0$ sufficiently large we can thus apply Lemma \ref{s2.2} to the integration along  lines parallel to the direction $e$ in the integral defining $J_k(\xi)$ near the point $x^0.$ Applying Fubini's theorem and  a partition of unity argument, we thus obtain
\begin{eqnarray}\label{estjk}\nonumber
 |J_k(\xi)|&\le& C\|\eta\|_{C^3(\RR^2)}\, 2^{-k|\ka|}(1+2^{-k}|\xi_3|)^{-1/j}\\
 &\le& C\|\eta\|_{C^3(\RR^2)}\, 2^{-k|\ka|}(1+2^{-k}|\xi|)^{-1/m},
\end{eqnarray}
where $m$ denotes the maximal  $j$ that arises in this context.

Summation in $k$ then yields the following estimates:
\begin{equation}\label{2.3}
|J(\xi)|\le C \|\eta\|_{C^3(\RR^2)}\,\left\{  \begin{array}{cc}
 (1+|\xi|)^{-1/m}, & \mbox{   if } m|\ka|>1\ ,\hfill\\
\log(2+|\xi|)\,(1+|\xi|)^{-1/m}, & \mbox{   if } m|\ka|=1\ ,\hfill\\
(1+|\xi|)^{-|\ka|}, & \mbox{   if } m|\ka|<1\
.\hfill
\end{array}\right.
\end{equation}

Now, if $h(\phi)=2$ and if the coordinates are not adapted to $\phi,$ then $m=m(\phi_\pr)>d(\phi)=1/|\ka|,$ so the first case in \eqref{2.3} applies. This implies \eqref{1.1}, in view of Lemma \ref{s2.1} (ii).

 Next, assume that the coordinates are adapted.
 If the principal face $\pi(\phi)$ is a compact edge, then $\phi_\ka=\phi_\pr,$ hence $1/|\ka|=d(\phi)=h,$ and moreover $m\le h.$ This implies $|\ka |m\le 1. $ Since in this case,  by Lemma \ref{vertex}, $\nu(\phi)=1$ if and only if $m=m(\phi_\pr)=h(\phi),$ i.e., if and only if $m|\ka|=1,$ we again obtain  estimate \eqref{1.1}.

 If $\pi(\phi)$ is unbounded, then $m=h$ and $1/|\ka|<h,$ so that the first case in
 \eqref{2.3} applies and we again verify \eqref{1.1}.

 Finally, if $\pi(\phi)$ is a vertex, then $1/|\ka|=h=m,$ so that the second case in \eqref{2.3} applies and we obtain \eqref{1.1} also in this case.

\subsection{The case of non-adapted coordinates: the contribution of regions   away from the principal root jet}\label{away}
\medskip

\bigskip
Assume next that the coordinates $x$ are not adapted to $\phi$ and that  $h>2$.

As we already explained in Section \ref{intro}, based on  Varchenko's algorithm we can then    locally  find a smooth real-valued function     $\psi$ which defines an adapted  coordinate system
\begin{equation}\label{2.4}
y_1:= x_1, \ y_2:= x_2-\psi(x_1)
\end{equation}
 for  the function $\phi$  near the origin. In these  coordinates, $\phi$ is given by
$$ \pad(y):=\phi(y_1,y_2+\psi(y_1)).$$

 In the case where Varchenko's algorithm stops after a finite number of steps because the principal face $\pi(\pad)$ is a compact edge and $m(\pad_\pr)=d(\pad),$ we meet the following {\it convention}:
 \smallskip

 We assume that we have then run the algorithm one further step (as in the proof of the implication (b) $\Rightarrow$ (a) in Lemma \ref{vertex}),  so that we may assume that $\pi(\pad)$ is a  vertex, i.e., the point $(h,h).$
{\it  Under this convention, $\pi(\pad)$ will be a vertex whenever $\nu(\phi)=1.$}

 \smallskip
 Consider the Taylor series 
\begin{equation}\label{2.5}
\psi(x_1)\approx \sum_{l\ge 1} b_lx_1^{m_l}
\end{equation}
of $\psi,$ where the $b_l$ are assumed to be  non-zero.
 After applying a linear change of coordinates, if necessary, we may and shall  assume that $b_1\ne 0$ and that the $m_l\in\NN$ form a strictly increasing sequence
$$2\le m_1<m_2<\cdots.$$

Suppose that the vertices of the Newton diagram $\N_d(\pad)$ of $\pad$ are  the   points
$(A_l,B_l), \  l=0,\dots,n,$ so that the Newton polyhedron $\N(\pad)$ is the convex hull
of the set $\bigcup_{l} ((A_l,B_l)+\bR_+^2),$ where $A_l<A_{l+1}$ for every $l\ge 0.$

\medskip

Let $L_l:=\{(t_1,t_2)\in \bN^2:\ka^l_1t_1+\ka^l_2 t_2=1\}$ denote the line passing through the  points $(A_{l-1},B_{l-1})$ and $(A_l,B_l),$ and let $a_l:={\ka^l_2}/{\ka^l_1}.$
 The $a_l$ can be identified as the distinct leading exponents of all the roots of $\pad$ in case that $\pad$ is analytic (see Section 3 of \cite{IKM-max}), and the cluster of roots whose  leading exponent in their Puiseux series expansion is given by  $a_l$ is associated to the edge $\ga_l:=[(A_{l-1},B_{l-1}) ,(A_l,B_l)]$ of $\N(\pad).$

As in Subsection 8.2 of \cite{IKM-max}, we choose the integer $l_0\ge 1$ such that
 $$
a_0<\dots <a_{l_0-1}\le m_1<a_{l_0}<\dots< a_l<a_{l+1}<\dots<a_n.
$$
 As has been shown in Section 3 of \cite{IKM-max},
the vertex  $(A_{l_0-1},B_{l_0-1})$ lies  strictly above the bisectrix, i.e.,  $A_{l_0-1}<B_{l_0-1},$ since the original coordinates $x$ were assumed to be non-adapted.

\medskip
Distinguishing the cases listed below, we single out a particular edge by fixing the corresponding index $\la\ge l_0$ as in Section  3 of \cite{IKM-max}:
\medskip

\noi{\bf Cases:}
\begin{enumerate}
 \item[(a)] In case (a), where the principal face of $\pad$ is a compact edge, we choose $\la$ so  that the edge $\ga_\la=[(A_{\la-1},B_{\la-1}) ,(A_\la,B_\la)]$ is the principal face $\pi(\pad)$ of the Newton polyhedron of $\pad.$

 \item[(b)] In  case (b), where $\pi(\pad)$ is the vertex $(h,h),$  we choose $\la$ so that  $(h,h)=(A_{\la},B_{\la}).$ Then $\la\ge 1,$ and $(h,h)$ is the right endpoint of the compact edge $\ga_\la.$

\item[(c)] Consider finally  case (c), in which the principal face $\pi(\pad)$ is unbounded, namely a half-line given by  $t_1\ge \nu_1$ and  $t_2=h,$ where $\nu_1<h.$  Here, we distinguish two sub-cases:
\begin{enumerate}

 \item[(c1)] If  the point $(\nu_1,h)$ is the right endpoint of a compact edge of $\N(\pad)$,  then  we choose again $\la\ge 1$ so that  this edge is given by  $\ga_\la.$

   \item[(c2)] Otherwise,  $(\nu_1,h)=(A_0,B_0)$ is the only vertex of $\N(\pad),$ i.e., $\N(\pad)=(\nu_1,h)+\RR^2_+.$  In this case, there is no non-trivial root $r,$ hence $n=0.$
   \end{enumerate}
\end{enumerate}

In the cases (a), (b) and (c1), let us put
\begin{equation}\label{aa}
a:=a_\la=\frac { \ka^\la_2}{  \ka^\la_1}.
\end{equation}
 We shall  assume in the sequel  that $\phi$ is analytic, since the general case can be reduced to this case as in \cite{IKM-max}.

\bigskip

 In a first step,
 we now decompose $J(\xi)=J^{1-\rho_1}(\xi)+J^{\rho_1}(\xi),$ where $\rho_1$ is the cut-off function introduced in Subsection 8.1 of \cite{IKM-max}  which localizes to a narrow $\ka$-homogeneous neighborhood of the form

\begin{equation}\label{restdomain}
|x_2-b_1x_1^{m_1}|\le \ve_1 x_1^{m_1}
\end{equation}
of the curve $x_2=b_1x_1^{m_1}.$ \begin{lemma}\label{snear}
Let $\ve_1>0.$ If   the neighborhood $\Om$ of the point $(0,0)$ is chosen sufficiently small,  then
$J^{1-\rho_1}(\xi)$ satisfies estimate \eqref{1.1}.

Moreover, if  $\N(\pad)$ is of the form $(\nu_1,h)+\RR^2_+,$  with $\nu_1<h,$  (case (c2) above), then the same statement holds true for $J(\xi)$ in place of $J^{1-\rho_1}(\xi).$
\end{lemma}

\proof

The oscillatory integral $J^{1-\rho_1}(\xi)$ can be estimated in a similar way as in the case of adapted coordinates by means of   Lemma \ref{s2.2}, and no  logarithmic factor is needed.
The reason for this is that any root of $\phi_\pr$ which does not agree with the principal root $x_2=b_1x_1^{m_1}$ has multiplicity strictly less then $d(\phi),$ as can be seen from Corollary 2.3 in \cite{IM-ada}, so that the third case of \eqref{2.3} applies.

\smallskip
Moreover, if $\N(\pad)=(\nu_1,h)+\RR^2_+,$ with  $\nu_1<h,$  then we recall from the proof of Lemma 8.1 in \cite{IKM-max}  that
$\phi_\ka(x)=cx_1^{\nu_1} (x_2-b_1 x_1^{m_1})^h,$  which implies  $h(\phi_\ka)=h(\pad_\ka)=h,$ and  we see that in this case we can again apply Lemma \ref{s2.2} to the $x_2$-integration in order to see that also $J^{\rho_1}(\xi)$ satisfies  \eqref{1.1}, without logarithmic factor, since here $1/|\ka |=d_h(\phi)<h.$ This proves also the second statement in the lemma.
\qed

\bigskip
 We may and shall therefore from now on assume that the Newton polyhedron of $\pad$ has at least one compact edge ''lying above'' the principal face, i.e., that one of the cases (a), (b)  or (c1)  applies. There remains $J^{\rho_1}(\xi)$ to be considered.
 \medskip

 In a next step we shall  narrow down the domain \eqref{restdomain} to a neighborhood of  the principal root jet of the form
\begin{equation}\label{restdomain2}
 |x_2-\psi(x_1)|\le N_\la  x_1^{a_\la},
\end{equation}
where $N_\la$ is a  constant to be chosen later. This domain is
 $\ka^\la$-homogeneous in the adapted coordinates $y.$ More precisely, we  fix a cut-off function $\rho\in C_0^\infty (\RR)$ supported in a  neighborhood of the origin such that $\rho=1$ near the origin, and put
$$
\rho_\la\x:=\rho\Big(\frac{ x_2-\psi (x_1)}{N_\la x_1^{a}}\Big).
$$

 \begin{prop}\label{s11.n}
Let  $N_\la>0.$  If   the neighborhood $\Om$ of the point $(0,0)$ is chosen sufficiently small,  then the oscillatory integral
$J^{1-\rho_\la}(\xi)$   satisfies estimate \eqref{1.1}.

Moreover, if the principal face $\pi(\pad)$ is  a vertex or unbounded, then the same holds true for $J(\xi)$ in place of $J^{1-\rho_\la}(\xi).$
\end{prop}

\proof

 To prove the first statement in the proposition, we decompose the difference set of the  domains \eqref{restdomain} and \eqref{restdomain2} in a similar way as in Subsection 8.2 of \cite{IKM-max} into   domains
 $$
D_l:=\{\x:\ve_l  x_1^{a_l}< |x_2-\psi(x_1)|\le N_l x_1^{a_l}\},\quad l=l_0,\dots,\la-1,
$$
which are $\ka^l$-homogeneous in the adapted coordinates $y$ given by \eqref{2.4},  and  intermediate  domains
$$E_l:=\{\x:N_{l+1} x_1^{a_{l+1}}<|x_2-\psi(x_1)| \le \ve_l x_1^{a_l}\}, \quad l=l_0,\dots,\la-1,
$$
and
$$E_{l_0-1}:=\{\x:N_{l_0} x_1^{a_{l_0}}<|x_2-\psi(x_1)| \le \ve_1 x_1^{m_1}\}.
$$
Here,  the $\ve_l>0 $ are  small and the $N_l>0$ are  large parameters to be determined later.
 \smallskip

 Notice that what will remain is the domain in   \eqref{restdomain2}. Deviating somewhat from our previous notation for $l<\la$  (and the one in \cite{IKM-max}),  we shall  denote this  domain by $D_\la,$ i.e.,
$$
D_\la:=\{ \x:|x_2-\psi(x_1)|\le N_\la x_1^{a_\la}\}.
$$


 \smallskip
 The localizations to these domains will be performed by means of cut-off functions
  $$
\rho_l\x:=\rho\Big(\frac{ x_2-\psi(x_1)}{N_lx_1^{a_l}}\Big)
-\rho\Big(\frac{ x_2-\psi(x_1)}{\ve_lx_1^{a_l}}\Big), \quad l=l_0,\dots,\la-1,
$$

$$
\tau_l\x:=\rho\Big(\frac{ x_2-\psi(x_1)}{\ve_lx_1^{a_l}}\Big)\, (1-\rho)
\Big(\frac{ x_2-\psi(x_1)}{N_{l+1}x_1^{a_{l+1}}}\Big), \quad l=l_0,\dots,\la-1,
$$
and
$$
\tau_{l_0-1}\x:=\rho\Big(\frac{ x_2-\psi(x_1)}{\ve_1 x_1^{m_1}}\Big)\, (1-\rho)
\Big(\frac{ x_2-\psi(x_1)}{N_{l_0}x_1^{a_{l_0}}}\Big),
$$
respectively by   $\rho_\la$ for the domain $D_\la.$

Here, in each instance $\rho\in C_0^\infty (\RR)$ is  a suitable cut-off function supported in the interval $[-1,1]$    such that $\rho=1$ near the origin.
Accordingly, we decompose
$$
J^{\rho_1}(\xi)-J^{\rho_\la}(\xi)=\sum_{l=l_0}^{\la-1} J^{\rho_l}(\xi)+\sum_{l=l_0-1}^{\la-1} J^{\tau_l}(\xi).
$$
The first part of Proposition \ref{s11.n} will be verified if  we  show that each of the oscillatory integrals $J^{\rho_l}(\xi)$ and $J^{\tau_l}(\xi)$ arising in this  sum satisfies estimate \eqref{1.1}.

\medskip
Now, the estimates of   Section 11 in \cite{IKM-max} show that the  $J^{\tau_l}(\xi)$ satisfy estimate \eqref{1.1}, even without logarithmic factor, so we  only need to consider the $J^{\rho_l}(\xi).$
\medskip

\noi {\bf Estimation of $J^{\rho_l}(\xi)$ for $l_0\le l\le \la-1.$} Applying the change of coordinates \eqref{2.4},  performing  a dyadic decomposition and re-scaling similarly as in the case of adapted coordinates, only with the weight $\ka$ replaced
by the weight $\ka^l,$ we find that
$$J^{\rho_l}(\xi)=\sum_{k=k_0}^\infty J_k(\xi),
$$
 where
\begin{eqnarray*}
J_k(\xi)
= 2^{-k|\ka^l|} \int_{(\RR_+)^2}e^{i\Big(2^{-k}\xi_3\phi^k(y)+2^{-k\ka^l_1}\xi_1y_1+2^{-k\ka^l_2}\xi_2y_2+2^{-k\ka^l_2}\xi_2\psi^k(y_1)\Big)}\rho^a_l(y) \,\eta^a(\de^l_{2^{-k}}y)\chi(y)\, dy,
\end{eqnarray*}
with $\psi^k(y_1):=2^{k\ka_2^l}\psi(2^{-k\ka^l_1} y_1),$ $ \phi^k(y):=\pad_{\ka^l}(y)+2^k\pad_r(\de^l_{2^{-k}} y),$  and
  $$
\rho^a_l(y):=\rho\Big(\frac{ y_2}{N_ly_1^{a_l}}\Big)
-\rho\Big(\frac{ y_2}{\ve_ly_1^{a_l}}\Big).$$
Here,   $\de^{l}_r$ denotes the dilation by $r>0$ associated to the weight $\ka^l,$ and  we have again decomposed
$$\pad=\pad_{\ka^l}+\pad_r,$$
where $\pad_r$ depends in fact also on $l$ and consists of terms of $\ka^l$-degree higher than $1.$

\medskip
Since  $l\le \la-1$ we can then again estimate $J_k(\xi)$ by means of Lemma \ref{s2.2}, applied to the $y_2$-integration, by using  Corollary 3.2 (i)  in \cite{IKM-max}, and obtain
\begin{eqnarray}\label{2.10n}\nonumber
 |J_k(\xi)|&\le& C\|\eta\|_{C^3(\RR^2)}\, 2^{-k|\ka^l|}(1+2^{-k}|\xi_3|)^{-1/d_h(\pad_{\ka^l})}\\
 &\le& C\|\eta\|_{C^3(\RR^2)}\, 2^{-k|\ka^l|}(1+2^{-k}|\xi_3|)^{-1/h}
 \end{eqnarray}
since $d_h(\pad_{\ka^l})<h.$ This also implies $1=|\ka^l|d_h(\pad_{\ka^l})|<|\ka^l| h,$ so that a comparison with \eqref{2.3} shows that   summation over  $k$ yields
$$
|J^{\rho_l}(\xi)| \le C\,\|\eta\|_{C^3(\RR^2)}\,(1+|\xi|)^{-1/h}.
$$

\medskip
We next turn to the second statement in Proposition \ref{s11.n}. We have to show that also $J^{\rho_\la}(\xi)$ satisfies estimate \eqref{1.1}. However,  if  the principal face $\pi(\pad)$ is  a vertex (case (b)) or unbounded (case (c1)) then  Corollary 3.2 (ii)  in \cite{IKM-max} allows us to argue  exactly as  before  in order to see that \eqref{2.10n} also holds for $l=\la.$ And, in case (b) we have $|\ka_\la|h=1,$ whereas in case (c1) we have $|\ka_\la|h>1,$ so that a comparison with \eqref{2.3} shows that estimate \eqref{1.1} is indeed valid for $J^{\rho_\la}(\xi).$
\qed

\subsection{The contribution of the homogenous domain $D_\la$ containing the  principal root jet}\label{near}

In view of Proposition \ref{s11.n}, we may and shall from now on assume that the principal face of $\N(\pad)$ is a compact edge (case (a)).
What remains to be estimated is the contribution of the  domain  \eqref{restdomain2} to $J(\xi),$ i.e., we are left with the oscillatory integral $J^{\rho_\la}(\xi).$ This will require  different arguments then those used   in \cite{IKM-max}.
We are also assuming that $x_1>0.$
Recall also that according to our convention
\begin{equation}\label{convb}
m(\pad_\pr)<d(\pad)=h,
\end{equation}
so that $\nu(\phi)=0.$ This means that we have to prove that $J^{\rho_\la}(\xi)$ satisfies \eqref{1.1}, without the presence of a logarithmic factor.

 \medskip

\subsubsection{Preliminary reductions}\label{reduc}
Following \cite{IKM-max}, Section 9,  it will be convenient at this point to defray our notation by writing $\phi$ in place of $\pad$ and $\eta$ in place of $\eta^a,$ $\ka$Ê in place of $\ka^\la,$  $\de_r$ in place of $\de^\la_r,$   etc.. With some slight abuse of notation, we shall denote $J^{\rho_\la}(\xi)$ by $J(\xi).$

After applying the change of coordinates $\eqref{2.4},$ this means that from now  on we shall  have to consider oscillatory integrals
$$
J(\xi):=\int_{\RR_+^2} e^{i\Big(\xi_3\phi(x)+\xi_1x_1+\xi_2(x_2+\psi(x_1))\Big)}\rho\Big(\frac{ x_2}{N_0x_1^{a}}\Big)\eta(x)\,dx,
$$
where $a=\ka_2/\ka_1>m_1,$ and where $N_0$ is a given, possibly large positive number.  Notice that the integration takes place only over the domain
\begin{equation}\label{2.10}
|x_2|\le N_0 x_1^{a}.
\end{equation}
 We shall write $m:=m_1,$  so that $\psi$ can be factored as
$\psi(x_1)=x_1^{m}\si(x_1),$ with a smooth function $\si$ satisfying  $\si(0)\ne 0.$  $J(\xi)$ can thus be written as an oscillatory integral
\begin{equation}\label{7.3}
J(\xi)=\int_{\RR_+^2} e^{iF(x,\xi)}\rho\Big(\frac{ x_2}{N_0x_1^{a}}\Big)\eta(x)\,dx,
\end{equation}
with a phase function
$$
F(x,\xi):=\xi_3\phi(x)+\xi_1x_1+\xi_2 x_1^m\si(x_1)+\xi_2x_2
$$
depending on $\xi\in\RR^3.$ The coordinates $x$ are now adapted to $\phi.$  We shall again decompose
$$
\phi(x)=\phi_\ka+\phi_r,
$$
where $\phi_\ka$ consists of terms of $\ka$-degree strictly bigger then $1,$ the $\ka$-degree of $\phi_\ka.$

\medskip
In order to estimate $J(\xi),$ in a first step we shall  decompose the domain \eqref{2.10} into smaller, $\ka$-homogeneous sub-domains. To this end, given any point $c\in[0,N_0],$ we define
$$
J^c(\xi):=\int_{\RR_+^2} e^{iF(x,\xi)}\rho\Big(\frac{ x_2-cx_1^a}{\ve_0x_1^{a}}\Big)\eta(x)\,dx,
$$
where $\ve_0>0$ will be a sufficiently small constant (the cut-off function $\rho$ here is possibly  different from the one in \eqref{7.3}).

In order to prove that $J(\xi)$ satisfies estimate \eqref{1.1}, it will be sufficient to show that for every $c\in[0,N_0]$ there exists an $\ve_0>0$ such that $J^c(\xi)$ satisfies  \eqref{1.1}, as can be seen easily be means of a partition of unity argument.

\medskip
We therefore assume that $c$ is fixed. Then we take again a smooth cut-off function $\chi$ which is supported in  an annulus $D$ such that
$$
\sum_{k=k_0}^\infty \chi(\de_{2^k}(x))=1\ \mbox{for every} \ x \in \supp\eta\setminus \{0\}.
$$
 Notice that we can assume that $k_0$ is a sufficiently large positive integer by choosing the  support of  $\eta$ sufficiently small.
Then we have
$$
J^c(\xi)=\sum_{k=k_0}^\infty J_k(\xi),
$$
where
$$
J_k(\xi):=\int_{\RR_+^2} e^{iF(x,\xi)}\rho\Big(\frac{ x_2-cx_1^a}{\ve_0x_1^{a}}\Big)\eta(x)\, \chi(\de_{2^k}(x))\,dx.
$$
After the change of variables $x\mapsto \de_{2^{-k}}(x)$ we obtain

\begin{equation}\label{jk}
J_k(\xi)=2^{-|\ka|k}\int e^{i2^{-k}\xi_3F_k(x,s)} \rho\Big(\frac{ x_2-cx_1^a}{\ve_0x_1^{a}}\Big)\eta(\de_{2^{-k}}(x))\,\chi(x)\,dx,
\end{equation}

where
 \begin{eqnarray*}
F_k(x,s)&:=&\phi_\ka(x)+2^k\phi_r(\de_{2^{-k}}(x))+s_1x_1+S_2x_1^{m}\si(2^{-\ka_1k}x_1)+s_2x_2,\\
s_1&:=&2^{(1-\ka_1)k}\frac{\xi_1}{\xi_3},\ s_2:=2^{(1-\ka_2)k}\frac{\xi_2}{\xi_3},\
S_2:=2^{(\ka_2-m\ka_1)k}s_2,\\
s&:=&(s_1,s_2,S_2).
\end{eqnarray*}
Note that $2\le m<a=\ka_2/\ka_1$ and $k\gg  1,$ so that $|S_2|\gg  |s_2|.$ Observe also that there exists a compact interval $I$  such that $x_1\sim 1$ on $I,$ so that  for every $\x$ in the support of the integrand of $J_k(\xi)$ as given by \eqref{jk},  we have
$$
x_1\in I \ \mbox{and} \ |x_2-cx_1^a|\lesssim \ve_0.
$$
Recall also  from \eqref{2.1} that we are  assuming  that $ \ |\xi|\sim |\xi_3|.$

\subsubsection{Estimation of the   oscillatory integrals $J_k(\xi)$ }\label{Jk}

In order to estimate $J_k(\xi),$ we shall distinguish several cases depending on the size of the parameters $s_1,s_2$ and $S_2.$ Recall here that $\xi$ is a function of $\xi_3, s_1, s_2 $ and $S_2.$

\medskip
{\bf Case 1.} $|S_2|  \ge M$ for some sufficiently  large constant $M\gg 1.$
In this case we can apply Lemma \ref{s2.2}  to the $x_1$-integration  and obtain
\begin{equation}\label{2.12}
|J_k(\xi)|\le C\frac{2^{-k|\ka|}\|\eta\|_{C^1}}{(1+2^{-k}|\xi|)^{1/2}}.
\end{equation}

\medskip
{\bf Case 2.} $|S_2|  < M,$ where $M$ is chosen as in Case 1.  Then $|s_2|\ll1,$ provided we have chosen $k_0$ sufficiently large.

If we  assume that  there is some $j\ge 1$ such that
\begin{equation}\label{2.13}
\partial_2^j\phi_\ka(1,c)\ne 0,
\end{equation}
then we claim that
\begin{equation}\label{2.14}
|J_k(\xi)|\le C\frac{2^{-k|\ka|}\|\eta\|_{C^1}}{(1+2^{-k}|\xi|)^{1/j}}.
\end{equation}
Indeed, by the homogeneity of $\phi_\ka,$ if we choose $\ve_0$ sufficiently small,  then $\partial_2^j\phi_\ka(x)\ne 0$ at every point $x$  in the support of the integrand of $J_k(\xi),$ so that the estimate follows for $j\ge 2$ from Lemma \ref{s2.2} again, this time applied to the $x_2$-integration.  Notice that the term $2^k\phi_r(\de_{2^{-k}}(x))$ can be viewed as a perturbation term.  Similarly, if $j=1,$ the estimate follows by an integration by parts with respect to $x_2.$

\medskip
{\it We  observe that if \eqref{2.13} holds for some $1\le j<h,$ then by \eqref{2.12}, \eqref{2.14}  and \eqref{2.3} we obtain the desired estimate \eqref{1.1}, even  without a logarithmic factor, since $h>2.$ }
\medskip

We may and shall therefore henceforth assume that
\begin{equation}\label{2.15}
\partial_2^j\phi_\ka(1,c)= 0 \ \mbox{for}\ 1\le j<h.
\end{equation}

\medskip
Recall that we are assuming that the  principal face of $\N(\phi)$  is a compact edge, so that  $\phi_\ka=\phi_\pr$ and $h= 1/{|\ka|}.$

\medskip
{\bf Assume first that $c=0.$}
Then necessarily  $\phi_\pr(1,0)\ne  0,$ for otherwise  $\phi_\pr$ would have  a root of multiplicity  at least $h$ at $(1,0),$ which would contradict \eqref{convb}.

Assuming  without loss of generality that   $\phi_\pr(1,0)=1,$
we  can then write (compare \cite{IKM-max}, Subsection 9.1)
$$
\phi_\pr\x=x_2^BQ\x+x_1^n,
$$
where $Q$ is a $\ka$-homogeneous polynomial such that
$Q(1,0)\ne 0,$ and where $B\ge h>2.$

\bigskip
Recall that  $|S_2|<M,$  so that  $|s_2|\ll1.$ We now distinguish two subcases:

\medskip
{\bf   Case 2.a.} $|S_2|  < M,$  and $|s_1|  \ge N$ for some sufficiently  large constant $N\gg 1.$
Then an integration by parts in $x_1$ leads to the estimate
$|J_k(\xi)|\le C\frac{2^{-k|\ka|}\|\eta\|_{C^1}}{1+2^{-k}|\xi|},$ which in return implies \eqref{1.1}, even without logarithmic factor.

\medskip
{\bf   Case 2.b.} $|S_2|  < M,$  hence $|s_2|\ll1,$ and $|s_1| < N,$ where $N$ is chosen as in Case 2.a.

We shall show that, given any point $(s_1^0,S_2^0)\in   [-M,M]\times [-N,N]$ and any point $x_1^0\in I,$ there exist a  neighborhood $U$ of
$(s_1^0,S_2^0),$ a neighborhood $V$ of $x_1^0$ and some $\om >1/h$ such that we have an estimate of the form
\begin{equation}\label{2.17}
|J_k(\xi)|\le C\frac{2^{-k|\ka|}\|\eta\|_{C^1}}{(1+2^{-k}|\xi|)^{\om}}
\end{equation}
for every $(s_1,S_2)\in U,$ provided the function $\chi$ in the definition of $J_k(\xi)$ is supported in $V$ and $\ve_0$ and $k$ are chosen sufficiently small, respectively large.  The same type of estimate will then hold also for every $(s_1,S_2)\in   [-M,M]\times [-N,N]$ and for the original function $\chi$ in the definition of $J_k(\xi),$ as can be seen by means of a partition of unity argument.   Summing  over all $k,$ this will clearly imply the estimate \eqref{1.1}, even without logarithmic factor.

\medskip
To this end, first notice that for $(s_1,S_2)\in U$  and $k$ sufficiently large, the function  $F_k(x,s)$ can be viewed as a small $C^\infty$- perturbation of the function
$$
F_{\pr}(x):=x_2^BQ\x+s^0_1x_1+S^0_2 \si(0)x_1^m+x_1^n.
$$
Thus, if $\nabla F_\pr(0,x_1^0)\ne 0,$ then we obtain \eqref{2.17}, with $\om=1,$ simply by integration by parts.

\medskip
Assume therefore that $(0,x_1^0)$ is a critical point of $F_\pr.$ Then $x_1^0$ is a critical point of  the polynomial function
$$
g(x_1):= s_1^0x_1+S_2^0 \si(0)x_1^m+x_1^n,
$$
 which comprises all terms of $F_\pr$ depending on the variable $x_1$ only. Note that $2\le m<n,$ since
 $n=1/\ka_1>\ka_2/\ka_1>m.$  It is then easy to see that $g''$ and $g'''$ cannot vanish simultaneously at the given point $x_1^0\in I,$ so that
  there are positive constants $c_1,c_2>0$ and a compact neighborhood $V$ of $x_1^0$ such that
$$
c_1\le\sum^3_{j=2}|g^{(j)}(x_1)|\le c_2\quad\mbox{for every}\  x_1\in V.
$$
This implies an analogous estimate for the partial derivatives $\partial_{x_1}^jF_k(x_1,x_2,s)$ of $F_k,$ uniformly for $(s_1,S_2)\in U$   and $x_2$ satisfying \eqref{2.10}, provided we choose $U$ and $\ve_0$ sufficiently small.
Applying the van der Corput type estimate in Lemma \ref{s2.2}, we thus obtain the  estimate \eqref{2.17} with $\om=1/3,$  so that we are done provided $h>3.$ Notice also that if $g''(x_1^0)\ne 0,$  then by  the same type of argument we see that \eqref{2.17} holds true with $\om=1/2>1/h.$

\medskip

We may thus  finally  assume that $2<h \le 3,$ and that $g'(x_1^0)=g''(x_1^0)=0.$  In this case we have
$$
\frac1{\ka_1+\ka_2}=h\le3 \quad \mbox {and}\quad \frac{\ka_2}{\ka_1}> m\ge2,
$$
so  that  $1/\ka_2< 9/2.$

Note that $B\le 1/\ka_2$ is a positive integer, and   $h\le B<9/2,$ so that either
 $B=4$ or $B=3$. We translate the critical point $(x_1^0,0)$ of $F_\pr$ to the origin by considering the function
 $$
 F_\pr^\sharp(y):=F_\pr(x_1^0+y_1,y_2)-g(x_1^0)=y_2^B Q(x_1^0+y_1,y_2)+
 \frac 16 g^{(3)}(x_1^0)\, y_1^3+\dots .
 $$
It is easy to see that this function has height $h^\sharp:=h(F_\pr^\sharp)$ given by
$h^\sharp=\frac 1{1/3+1/4}=12/7,$  if $B=4,$ and  $h^\sharp=\frac 1{1/3+1/3}=3/2,$ if
$B=3.$

In both cases, $F_\pr^\sharp$ has height $h^\sharp<2$ (indeed, according to  Arnol'd's classification of  singularities,  $F_\pr^\sharp$ is  of type  $E_6$ and $D_4,$ respectively). We can  therefore again apply
 Duistermaat's results in \cite{duistermaat} to  the oscillatory integral $J_k(\xi)$ and obtain the estimate \eqref{2.17}, with $\om=1/h^\sharp>1/h.$ Note here that the estimates in \cite{duistermaat} are stable under small perturbations.

\medskip
{\bf Assume finally that $c>0.$} Then, by Corollary 3.2 (iii)  in \cite{IKM-max}, our assumption  \eqref{2.15} implies that necessarily  $a=\ka_2/\ka_1\in\NN.$

 We  can then reduce this case to the previous case $c=0$  by performing  another change of variables $x_2\mapsto x_2+cx_1^a$ in the integral defining $J_k(\xi).$

Indeed, this  is equivalent to replacing the function $\psi$ in our previous argument by $\psi^\sharp(x_1):=\psi(x_1)+cx_1^a,$ and assuming that $c=0.$ Denote by  $\phi^\sharp\x:=\phi(x_1, x_2+cx_1^a)$ the corresponding phase function. Then the  coordinates $\x$ are adapted to $\phi^\sharp$ too, as can be seen as follows:

 Lemma 3.1 in \cite{IM-ada} shows that   $(\phi_\pr)^\sharp\x:= \phi_\pr(x_1, x_2+cx_1^a)$ is again a $\ka$-homogeneous polynomial whose principal face intersects the bi-sectrix, and
$m(\phi_\pr)=m((\phi_\pr)^\sharp).$ Therefore $(\phi_\pr)^\sharp$ must be the principal part of $\phi^\sharp.$

\medskip

This completes the proof of Theorem \ref{s1.1}.

\section{Sharpness of the  uniform estimates }

In this section, we shall give a proof of Theorem \ref{limit}. Observe that we may assume for this purpose that the coordinates $\x$ are adapted to $\phi,$ so that $d:=d(\phi)=h.$ 

\smallskip
We shall only consider the asymptotic behavior of $J_+(\la),$ since the result for $J_{-}(\la)$ follows from the one for $J_+(\la)$ by means of complex conjugation.

\smallskip
\begin{remarks}\label{arnold}
{\rm
If $h<2,$ then   the phase function $\phi$ has a critical point at the origin with finite Milnor number,  and can thus be reduced to a  polynomial phase function by means of a smooth local change of coordinates at the origin (see \cite{arnold}). Therefore, in this case we could  apply  the  classical results for analytic phase functions by A.N. Varchenko  \cite{Va}. However, we will give a more elementary argument which does not rely on this classification of singularities.

Notice also that if  $h=1,$  then the phase function has a non-degenerate critical point at
the origin in our adapted coordinates,   and we could  apply the method of stationary phase in order to prove the existence of the limits in Theorem \ref{limit} (see \cite{stein-book}).We shall, however, proceed somewhat differently also in this case.}

\end{remarks}

\subsection{The case where the principal face is a compact edge}\label{compactedge}
We begin with the  simplest case where the principal face $\pi(\phi)$ is a compact edge. 
Arguing as in Subsection \ref{away}, we may then assume in addition that 
$$ m(\phi_\pr)<d,
$$
since otherwise a suitable local change of coordinates would reduce us to the situation where the principal face is a vertex.

Then there is a unique weight $\ka$ such that $\pi(\phi)$ is lying  on the line given by the equation $\ka_1t_1 +\ka_2t_2=1.$ Without loss of generality we may assume  that $0<\ka_1\le \ka_2$. 
Recall also that then $\phi_\pr=\phi_\ka$ and $d=d_h(\phi_\ka)=1/|\ka|,$ and that if we decompose 
$$
\phi(x)=\phi_\ka(x)+\phi_r(x),
$$
then $\phi_r$ is an error term whose Newton polyhedron is contained in the set $\{(k_1,\,k_2)\in \bZ^2: \, \ka_1 k_1+\ka_2 k_2>1 \}.$

\medskip 
In a first step, we shall reduce ourselves to the situation where the amplitude $a$ is constant on a neighborhood of the origin. To this end, if $\Om$ is an open neighborhood of the origin in $\RR^2,$  let us introduce the subspace of amplitude functions
$$
\dot C^3_0(\Om):=\{a\in C_0^3(\Om): \, a(0,0)=0\}.
$$

If $a\in \dot C^3_0(\Om)$ and if $F$ is a smooth, real-valued phase function on $\Om,$ we consider the oscillatory integral
$$
J(\la, F,a):=\int e^{i\la (\phi_\ka(x)+F(x))}a(x)\,dx, \qquad \la>0.
$$

\begin{prop}\label{outside}
There exists a positive number $\varepsilon$ such that for any smooth function $F\in C^\infty(\bR^2)$ with
$\N(F)\subset \{(k_1,\,k_2)\in \bZ^2: \, \ka_1 k_1+\ka_2 k_2>1 \}$ there exists  a neighborhood $\Om$ of the origin so that for any $a\in \dot C^3_0(\Om)$ the following estimate
$$
|J(\la,\, F, a)|\le \frac{C(F)\|a\|_{C^3(\Om)}}{\la^{1/d+\varepsilon}}
$$
holds true, with a constant $C(F)$ depending only on  the $C^N(\Om)$ norm of $F,$ for some sufficiently large  number $N$.
\end{prop}

\begin{proof}
If $a\in \dot C^2_0(\Om),$ then $a$  can be written as $a\x=x_1a_1\x+x_2a_2\x,$ with functions  $a_1,\, a_2\in C^1(\Om)$ whose $C^1$-norms can be controlled by the $C^2$-norm of $a.$
Consequently for the oscillatory integral we have
$$
J(\la,\, F, a)=J(\la,\, F, x_1a_1)+J(\la,\, F, x_2a_2).
$$
We shall therefore estimate $J(\la,\, F, x_1a_1)$ ($J(\la,\, F, x_2a_2)$ can be treated in a similar way). As before, we choose a suitable smooth cut-off function $\chi$ on $\RR^2$ supported in an annulus $D$ such that the functions $\chi_k:=\chi\circ \de_{2^k}$ form a partition of unity,  and then decompose 
 $$
J(\la,\, F, x_1a_1)=\sum_{k=k_0}^\infty J(\la,\, F, x_1a_1\chi_k).
$$
Here, $\de_r$ denotes again the dilation by the factor $r>0$ associated to the weight $\ka.$ 
Recall that by choosing $\Om$ sufficiently small we may assume that $k_0$ is a sufficiently large  number. 
After re-scaling, we may re-write
$$J_k(\la):=J(\la,\, F, x_1a_1\chi_k)
$$
as
$$
J_k(\la)=2^{-(|\ka|+\ka_1)k}\int e^{i\la 2^{-k}(\Phi_\ka+2^kF(\de_{2^k}(x)))} x_1a_1(\de_{2^{-k}}(x))\chi(x)\,dx.
$$

If $\la 2^{-k}\le M$ (with a fixed positive number $M$), a trivial estimate for the integral $J_k(\la)$ gives
$$|J_k(\la)|\le C \|a_1\|_{C^0(\Om)} 2^{-(|\ka|+\ka_1)k},$$ 
hence
\begin{equation}\label{5.3a}
\sum_{\la 2^{-k}\le M}|J_k(\la)|\le C_M\frac{\|a\|_{C^1(\Om)}}{\la^{|\ka|+\ka_1}},
\end{equation}
if we assume without loss of generality that $\Om$ is a ball.
\medskip

Assume next that $\la 2^{-k}> M.$
Since $\|2^kF\circ\de_{2^{-k}}\|_{C^m(\Om)}\to 0 \mbox{ as } k\to +\infty,$ by choosing $\Om$ sufficiently small we may assume  that
$\|2^kF\circ\de_{2^{-k}}\|_{C^m(\Om)}$ is  sufficiently small.
\medskip

Now, if $m(\phi_\ka)\ge 1,$  we put $m:=m(\phi_\ka).$  Then $1\le m<d.$ Notice that if $x^0\in D$ is such that $\nabla \phi_\ka(x^0)=0,$ then, by Euler's homogeneity relation, also $\phi(x^0)=0.$ Therefore, by applying Lemma \ref{s2.2}, respectively an integration by parts, and assuming that $M$ is sufficiently big, we see that 

$$
|J_k(\la)|\le C(F) \|a_1\|_{C^1(\Om)} 2^{-(|\ka|+\ka_1)k}(1+2^{-k}\lambda)^{-1/m}.
$$
Summing in $k,$ this implies
\begin{equation}\label{5.3}
\sum_{\la 2^{-k}> M}|J_k(\la)|\le C(F) \|a\|_{C^2(\Om)}\,\left\{  \begin{array}{cc}
 (1+\la)^{-1/m}, & \mbox{   if } m(|\ka|+\ka_1)>1\ ,\hfill\\
\log(2+\la)\,(1+\la)^{-1/m}, & \mbox{   if } m(|\ka|+\ka_1)=1\ ,\hfill\\
(1+\la)^{-(|\ka|+\ka_1)}, & \mbox{   if } m(|\ka|+\ka_1)<1\
.\hfill
\end{array}\right.
\end{equation}
If we put  $\varepsilon_0:=\min\{\ka_1, 1/m-1/d  \},$  we see that \eqref{5.3a} and \eqref{5.3} imply that
$$
|J(\la,\, F, x_1a_1)|\le \frac{C(F)\|a\|_{C^2(\Om)}}{\la^{1/d+\varepsilon}},
$$ for every  positive number $\varepsilon<\varepsilon_0.$ 
Similar estimates hold true for $J(\la,\, F, x_2a_2),$ only with $\ka_1$ replaced by $\ka_2.$ Since $\ka_1\le \ka_2,$ we see that we can use the same range of $\ve$'s also in this case and obtain the desired estimate in Proposition \ref{outside}.
\medskip

There remains the case where $m(\phi_\ka)=0.$ Here, $\phi_\ka$ does not vanish away from the origin, and thus $\nabla \phi_\ka(x^0)\ne 0$ for every $x^0\in D.$ Thus, choosing $m=1$ here and applying one integration by parts, we again obtain estimate \eqref{5.3}, and can conclude as before, if $d>1.$

Finally, if $d=1$ (notice that necessarily $d\ge 1,$ since $\nabla\phi(0,0)=0$), applying two integrations by parts to $J_k(\la),$ we obtain
$$
\sum_{\la 2^{-k}> M}|J_k(\la)|\le C(F) \|a\|_{C^3(\Om)}\, \la^{-(|\ka|+\ka_1)},
$$
where $|\ka|=1.$ Thus, we can choose $\ve:=\ka_1$ in this case.

\qed


Let us now consider  the oscillatory integral
$$
J_{+}(\la):=\int_{\RR^2}e^{ i\la \phi(x)}\eta(x)\, dx,
$$
where $\eta\in C_0^\infty(\Om).$ Choose a smooth bump function $\chi_0$ supported in $\Om$ which is identically $1$ on a neighborhood of the origin. Then, if we choose $\Om$ sufficiently small,  Proposition \ref{outside} implies  that  the oscillatory integrals $J_{+}(\la)$ and $\eta(0,0)J(\la),$ with 
$$
J(\la):=\int_{\RR^2}e^{ i\la \phi(x)}\chi_0(x)\, dx,
$$
differ by a term of decay rate $O(\la^{-1/d-\ve}).$ This shows that, in order to prove Theorem \ref{limit} in this case, it suffices to prove that the limit 
\begin{equation}\label{lim1}
\lim_{\la\to +\infty}\la^{1/d}J(\la)=c
\end{equation}
exists and that $c\ne 0.$

To this end, put $\de:=\ve/4,$ with $\ve>0$ as in  Proposition \ref{outside},  and define the polynomial functions  $P$ and $Q$  by
$$
Q(x):=\sum_{|\alpha|\le 1/\de+3} \frac{\partial^\alpha \phi(0)}{\alpha!}x^\alpha=:\phi_\ka(x)+P(x).
$$

Notice that all the derivatives of the function $e^{i\la (\phi(x)-Q(x))}$ up to order $3$ are uniformly bounded with respect to  $\la$   on the set where $\la^\de |x|<1$. We therefore decompose 
$$
J(\la)=\int e^{i\la\phi(x)}\chi_0(x)\chi_0(\la^\de x)\,dx+\int e^{i\la\phi(x)}\chi_0(x)(1-\chi_0(\la^\de x))\, dx.
$$
Due to Proposition \ref{outside} (with $F:=P$),  the second summand has decay  rate of order $O(\la^{-1/d-\ve+3\de})=O(\la^{-1/d-\de})$  as $ \la\to +\infty,$ if $\Om$ is supposed to be chosen sufficiently small. In order to prove \eqref{lim1}, we may therefore replace $J(\la)$ by the first summand, $J_0(\la),$ which we again decompose as 
\begin{eqnarray*} 
J_0(\la)&=&\int e^{i\la\phi(x)}\chi_0(x)\chi_0(\la^\de x)\,dx\\
&=& \int e^{i\la Q(x)}\chi_0(x)\,dx+
\int e^{i\la(\phi_\ka(x)+P(x))}\chi_0(x)\Big(\chi_0(\la^\de x) e^{i\la(\phi(x)-Q(x))}-1\Big)\, dx.
\end{eqnarray*}
Again, by applying Proposition \ref{outside}, we see that  the second summand has decay  rate 
$O(\la^{-1/d-\de})$ as  $\la\to +\infty,$ and thus we are reduced to proving that the limit
$$
\lim_{\la\to +\infty}\la^{1/d}\int e^{i\la Q(x)}\chi_0(x)\,dx=c
$$
exists and that $c\ne 0.$ But,  $Q(x)=\phi_\ka(x)+P(x)$ is a polynomial, 
with principal part $\phi_\ka,$ and therefore this statement follows from the classical results for analytic phase functions due to  Varchenko \cite{Va} (see also \cite{greenblatt}).

\subsection{The case where the principal face is a vertex}\label{pvertex}

Assume now that $\pi(\phi)=\{(d,d)\}$ is a vertex, so that in particular $d$ is a positive  integer. After multiplying the phase function with a suitable real constant (this can be implemented by  means of a suitable scaling in $\la$ and, possibly, complex conjugation of  $J_{+}(\la)$), we may  assume without loss of generality that the principal part of $\phi$ is given by 
$$\phi_\pr(x)=x_1^dx_2^d.$$

We may also assume that the coordinates $\x$ are super-adapted, in the sense of Greenblatt \cite{greenblatt}. Then, if $d=h=1,$ according to Lemma 1.0 in \cite{greenblatt},  the critical point of $\phi$ at the origin is non-degenerate, and thus the statement of Theorem \ref{limit} is a well-known consequence of the method of stationary phase.
\medskip

Let us therefore henceforth assume that $d=h\ge 2$.

 \medskip

\subsubsection{Two compact edges}\label{twoc}

First, assume that the Newton polyhedron  $\N(\phi)$ has two compact edges containing the vertex $(d,d)$ as one of their endpoints, say $\gamma_a,$ lying "above"  the  bi-sectrix and  on the line given by $\ka_1^a t_1+\ka_2^a t_2=1,$  and $\gamma_b,$ lying "below" the principal face and  on the line given by  $\ka_1^b t_1+\ka_2^b t_2=1.$ Notice that then
\begin{equation}\label{bigger}
a:=\frac{\ka_2^a}{\ka_1^a}<\frac{\ka_2^b}{\ka_1^b}=:b.
\end{equation}

\begin{lemma}\label{forma}
The function $\phi$ can be written as
$$
\phi\x=x_1^dx_2^d+\phi_a\x+\phi_b\x,
$$
where $\phi_a\x=x_2^d\tilde\phi_a\x$ and $\phi_b\x=x_1^d\tilde\phi_b\x,$ with smooth functions $\tilde\phi_a$ and  $\tilde\phi_b$.
\end{lemma}

The proof of Lemma \ref{forma} is straightforward. Notice also that  we have
$$
x_1^dx_2^d+\phi_a\x=\phi_{\ka^a}\x+\phi_{a,r}\quad \mbox{and}\quad x_1^dx_2^d+\phi_b\x=\phi_{\ka^b}\x+\phi_{b,r},
$$
where $\phi_{\ka^a}$ is $\ka^a:=(\ka_1^a,\ka_2^a)$-homogeneous of degree $1,$ and $\phi_{a,r}$ consists of terms of $\ka^a$-degree higher than $1,$ and where the analogous statements holds true for $\phi_{\ka^b}$ and $\phi_{b,r}.$

\begin{lemma}\label{supera}
After applying a suitable smooth local change of coordinates at the origin, we may 
assume that  the functions $x_1\mapsto\phi_{\ka^a}(x_1, \pm1)$ and  $x_2\mapsto\phi_{\ka^b}(\pm1,x_2)$
have no root of multiplicity greater or equal to   $d,$ respectively.
\end{lemma}

\begin{proof} We may assume that $b\ge 1,$ for otherwise, after interchanging  the coordinates $x_1$ and $x_2,$ we will have $b\ge a\ge 1.$

Then, the proof of Theorem 7.1 in \cite{greenblatt} for the existence of ``super-adapted coordinates'' shows that, after applying a suitable  local change of coordinates at the origin, we may assume that $\phi_{\ka^b}(\pm1,x_2)$ has no non-zero root of order greater or equal to $d$ (of course, the edge $\ga_b$ may have changed and even have become unbounded, but this would be a case to be considered later). We also remark that the change of coordinates in \cite{greenblatt} is such that the edge $\ga_a$ remains to be an edge of the Newton diagram in the new coordinates.

According to Proposition 2.2 in \cite{IM-ada}, we can then write, say for $x_1>0, $
$$
\phi_{\ka^b}(x_1,x_2)=x_1^{\al}x_2^{\beta}\prod_{l}(x_2^q-c_lx_1^p)^{n_l},
$$
for suitable integers $\al,\beta\ge 0$ and $p,q\ge 1$ such that  $p/q=b,$  where  $c_l\in \bC\setminus\{0\}$ and  $n_l\in \bN\setminus\{0\}$. Since we are assuming that $(d,d)$ is the upper vertex of the edge $\gamma_b,$ we see that $\al=d$ and $ \beta+(\sum_ln_l)q=d.$ Therefore, necessarily $\beta<d,$  which shows that $\phi_{\ka^b}(\pm1,x_2)$ that also $x_2=0$ is no root of order greater or equal to $d.$

\medskip We now turn to $\phi_{\ka^a}.$ If $a\le 1,$ after interchanging again the coordinates $x_1$ and $x_2,$ hence also the edges $\ga_a$ and $\ga_b,$ we may assume that $x_1\mapsto\phi_{\ka^a}(x_1, \pm1)$ has no root of multiplicity greater or equal to   $d,$ and that $a\ge 1.$ Applying then the previous argument again to $\ga_b,$ we see that in addition we may assume that   $\phi_{\ka^b}(\pm1,x_2)$
has no root of multiplicity greater or equal to   $d,$ and are done.

\smallskip Assume finally that $a>1.$ Then we can accordingly write 
$$
\phi_{\ka^a}(x_1,x_2)=x_1^{\al}x_2^{\beta}\prod_{l}(x_2^q-c_lx_1^p)^{n_l},
$$
where now $p/q=a.$ Since $(d,d)$ is the lower vertex of $\ga_a,$ we see that $\beta=d$ and  $ \al+(\sum_ln_l)p=d,$ hence $\al<d.$ Moreover, if $a\notin\NN,$ then Corollary 2.3 in \cite{IM-ada} shows that $n_l<d$ for every $l,$ which shows that    $\phi_{\ka^a}(x_1, \pm1)$
has no root of multiplicity greater or equal to   $d.$ 

And, if $a\in\NN,$ then $q=1$ and $p=a>1,$ hence $n_l<n_l p\le d,$ so that again $n_l<d,$ and we can conclude as before.
\end{proof}

Let us assume in the sequel that the adapted  coordinates  are chosen so that the conclusions in   Lemma \ref{supera} do apply, and  consider again the oscillatory integral
$$
J_{+}(\la):=\int_{\RR^2}e^{ i\la \phi(x)}\eta(x)\, dx.
$$
Note that in this case we have to prove that 
$$
\lim_{\la\to +\infty}\frac{\la^{1/d}}{\log{\la}}J_{+}(\la)=c\,\eta(0),
$$
where $c\ne 0.$

With $\chi_0$ as before, let us consider the  oscillatory integrals
$$
J_1(\la):=\int e^{i\la \phi(x)}(\eta(x)-\eta(0)\chi_0(x))\, dx
$$
and 
$$
J(\la):=\int_{\RR^2}e^{ i\la \phi(x)}\chi_0(x)\, dx.
$$
We then have the following substitute for  Proposition \ref{outside}, which allows to reduce to proving that the following limit 
 \begin{equation}\label{lim2}
\lim_{\la\to +\infty}\frac{\la^{1/d}}{\log{\la}}J(\la)=c
\end{equation} 
exists and is non-zero.

\begin{lemma}\label{simple1}
If $\Om$ is chosen sufficiently small, then the following estimate
$$
|J_1(\la)|\le \frac{C\|\eta\|_{C^2(\Om)}}{\lambda^{1/d}}
$$
holds true.
\end{lemma}

\begin{proof} Permuting the coordinates $x_1,x_2,$  if necessary, we may choose a weight $\ka=(\ka_1,\ka_2)$ with $0<\ka_1\le \ka_2,$ such that the line given by $\ka_1t_1+\ka_2t_2=1$ is a supporting line to $\N(\phi)$ which contains only the point $(d,d)$ of $\N(\phi).$ Arguing now in the same way is in the proof of Proposition \ref{outside}, with $m:=d,$ we obtain the desired estimate.

\end{proof}
Choose a smooth cut-off function $\chi^0\in C^\infty_0(\RR)$ supported in a sufficiently small neighborhood of the origin.
In order to prove \eqref{lim2}, let us decompose
$$
J(\la)=J_0(\la)+J_\infty(\la),
$$
where 
\begin{eqnarray}
&&J_0(\la):=\int e^{i\la \phi\x}\chi_0\x\chi^0\left(\frac{x_2}{\varepsilon |x_1|^{a}}\right)\chi^0\left(\frac{x_1}{\varepsilon |x_2|^{1/b}}\right) \,dx,\label{j0} \\
&&J_\infty(\la):=\int e^{i\la \phi\x}\chi_0\x\left(1-\chi^0\left(\frac{x_2}{\varepsilon |x_1|^{a}}\right)\chi^0\left(\frac{x_1}{\varepsilon |x_2|^{1/b}}\right)\right) \,dx,\label{j00}
\end{eqnarray}
where $\ve>0$ will be chosen later.

\begin{lemma}\label{simple2}
Let $\ve>0.$ Then, if  $\Om$ is chosen sufficiently small, the following estimate
$$
|J_\infty(\la)|\le \frac{C\|\eta\|_{C^2(\Om)}}{\lambda^{1/d}}
$$
holds true.
\end{lemma}
\begin{proof}

We decompose $J_\infty(\la)=J_a(\la)+J_b(\la),$ where
\begin{eqnarray*}
J_a(\la)&:=&\int e^{i\la \phi\x}\chi_0\x\left(1-\chi^0\left(\frac{x_2}{\varepsilon |x_1|^{a}}\right)\right) \,dx,\\
J_b(\la)&:=&\int e^{i\la \phi\x}\chi_0\x\left(1-\chi^0\left(\frac{x_1}{\varepsilon |x_2|^{1/b}}\right)\right)\chi^0\left(\frac{x_2}{\varepsilon |x_1|^{a}}\right) \,dx,
\end{eqnarray*}
and show that both terms separately satisfy the estimate in Lemma \ref{simple2}.

\medskip
We begin with $J_a(\la).$ Using the dilations $\de_r$ associated to the weight $\ka^a,$ we dyadically decompose $J_a(\la)=\sum_{k=k_0}^\infty J_k(\la)$ in a similar way as in the proof of Lemma \ref{outside}. Here, after re-scaling, $J_k(\la)$ is given by
$$
J_k(\la)=2^{-|\ka^a|k}\int e^{i\la 2^{-k}(\Phi_{\ka^a}+2^k\phi_r(\de_{2^k}(x)))} \chi_0(\de_{2^{-k}}(x))
\left(1-\chi^0\left(\frac{x_2}{\varepsilon |x_1|^{a}}\right)\right)\chi(x_1,x_2)\,dx,
$$
where $|\ka^a|=1/d.$ Notice that 
$$|x_1|\lesssim 1\mbox{   and  }\ve \lesssim |x_2|\lesssim  1$$
 for every $\x$ in the support of the integrand. Let $m$ denote the maximal order of vanishing of $\phi_{\ka^a}$ transversal to its roots on this domain. Then $m<d,$ since, according to Lemma \ref{supera}, we are assuming that $\phi_{\ka^a}(x_1, \pm1)$ has no root of order greater or equal to $d.$ Consequently, we have  $m|\ka^a|<1.$ Arguing as in the proof of Lemma \ref{outside} in order to estimate the $J_k(\la),$ and summing in $k,$  we then find that $|J_a(\la)|\le C \la^{-|\ka^a|}=C\la^{-1/d}.$

\smallskip
 Finally, $J_b(\la)$ can be estimated in a very similar way, making use of the dilations associated to the weight $\ka^b$ in place of $\ka^a.$ Note that the additional factor 
$\chi^0\Big({x_2}/({\varepsilon |x_1|^{a}})\Big)$ appearing in the integral defining $J_b(\la)$ is under control because of \eqref{bigger}.
\end{proof}

The proof of \eqref{lim2} is thus reduced to proving the next 
\begin{lemma}\label{main}
The following limit 
$$
\lim_{\la\to+\infty} \frac{\la^{1/d}}{\log{\la}} J_0(\la)
$$
exists  and is non-zero. Moreover, it  does not depend on the choice of $\ve .$
\end{lemma}

\begin{proof}
Let us first assume that the integer $d\ge 2$ even.

We may also assume without loss of generality that $\chi^0$ is an even function and that $\chi_0$ is radial, so that in particular
$$
\chi_0(x_1,\,x_2)= \chi_0(\pm x_1,\,\pm x_2).
$$
This implies that, if we decompose the integral defining $J_0(\la)$ into the four integrals over each of the  quadrants of $\RR^2,$ then, after an obvious change of coordinates, all four of them will have the same amplitude, as well as the same principal part $x_1^dx_2^d$ for their phases. Since  we shall see that the leading term in the asymptotic expansion of $J_0(\la)$ will only depend on the principal part of the phase function, we may thus reduce ourselves  to considering  the integral $J_0^+(\la)$ over the positive quadrant only. 
\smallskip

Notice that by \eqref{bigger} $b-a>0.$
In the integral for $J_0^+(\la)$ we apply the  change of variables
$$
x_2=x_1^a y_2,\quad x_1=y_2^\frac{1}{b-a}y_1,
$$
and denote by $\tilde\phi$ the phase function when expressed in the coordinates $y,$ i.e., $\tilde \phi(y)=\phi(x).$

Observe that this change of coordinates is of class $C^1$ on the closed positive quadrant, and of class $C^\infty$ away from the coordinate axes, and that it leads to the following form of the phase function $\tilde\phi:$
$$
\tilde\phi(y_1,y_2)=y_1^{d(1+a)}y_2^\frac{d(1+b)}{b-a}(1+\rho(y_1^\delta, y_2^\delta)),
$$
where $\rho(z_1, z_2)$ is a smooth function with $\rho(0, 0)=0,$ and where $\de=1/p>0$ is some  rational number.

Indeed, the Newton polyhedron is transformed into $\N(\tilde\phi)=(d,d)+\RR_+^2$ under this change of variables, and since 
$$
 x_1=y_2^\frac{1}{b-a}y_1,\quad x_2=y_1^a y_2^{1+\frac{a}{b-a}},
$$
it is clear that if $f$ is any smooth function of $x$ which is flat at the origin, i.e., which  vanishes to infinity oder at the origin, then $\tilde f,$ defined by $\tilde f(y)=f(x),$ can be factored as 
$\tilde f(y)=y_1^{d(1+a)}y_2^\frac{d(1+b)}{b-a}g(y),
$
 where also $g$ is smooth and flat at the origin.

The oscillatory integral  $J_0^+(\la)$ then transforms into
$$
J_0^+(\la)=\int e^{i\la\tilde\phi(y_1,y_2)}y_1^a y_2^\frac{1+a}{b-a}\chi_0\Big(\frac{y_1}{\varepsilon}\Big)\chi_0\Big(\frac{y_2}{\varepsilon}\Big)\tilde\chi_0(y_1,y_2)\,dy,
$$
where $\tilde\chi_0$ is of class  $C^1$ on the closed positive quadrant, and of class $C^\infty$ away from the coordinate axes, and $\tilde\chi_0(0,0)=1.$

Observe next that if $M$ is any positive constant, then the contribution to the integral $J_0^+(\la)$ by the sub-domain where $\la y_1^{d(1+a)}\le M$ is trivially of order
 $O(\la^{-1/d})$ as $ \la\to+\infty.$ 
 
 We may therefore consider the oscillatory integral
 \begin{eqnarray*}
I(\la)&:=&\int\limits_{\la y_1^{d(1+a)}>M}\int e^{i\la\tilde\phi(y_1,y_2)}y_1^a y_2^\frac{1+a}{b-a}\chi_0\Big(\frac{y_1}{\varepsilon}\Big)\chi_0\Big(\frac{y_2}{\varepsilon}\Big)\tilde\chi_0(y_1,y_2)\,dy_2 dy_1\\
&=&\int\limits_{\la y_1^{d(1+a)}>M}y_1^a\chi_0\Big(\frac{y_1}{\varepsilon} \Big)I_{\rm int}(\la, y_1)\, dy_1
\end{eqnarray*}
in place of $J_0^+(\la),$ where $M$ is a fixed, sufficiently large positive number. 

Assuming that $\ve>0$ is chosen sufficiently small, we may apply the  change of variables
$$
 z_2:=y_2(1+\rho(y_1^\delta, y_2^\delta))^\frac{b-a}{d(1+b)}
$$
to the inner integral  
$$I_{\rm int}(\la, y_1):=\int e^{i\la\tilde\phi(y_1,y_2)} \chi_0\Big(\frac{y_2}{\varepsilon}\Big)\tilde\chi_0(y_1,y_2)\,y_2^{\frac{1+a}{b-a}}\, dy_2,
$$
which leads to 
$$I_{\rm int}(\la, y_1)=\int e^{i\la y_1^{d(1+a)}z_2^\frac{d(1+b)}{b-a}} \chi_0\Big(\frac{z_2(1+\tilde\rho(y_1,z_2))}{\varepsilon}\Big)\tilde\chi_{0,0}(y_1,z_2)\,z_2^{\frac{1+a}{b-a}}\, dz_2,
$$
where $\tilde\rho$ and $\tilde\chi_{0,0}$ have similar properties as $\rho$ and $\tilde\chi_{0},$ respectively. Changing variables in this integral to $t:=z_2^\frac{1+b}{b-a},$ and applying some classical results on one-dimensional oscillatory integrals with  critical points  of order $d$ (see  A. Erde'lyi  \cite{Erde'lyi}, Section 2.9), we thus  obtain
$$
I_{\rm int}(\la, y_1)=\frac{b-a}{1+b}\, \Big(\frac{C_d}{(\la y_1^{d(1+a)})^{1/d}}+ R(\la, y_1)\Big),
$$
where $C_d\ne 0$ is given explicitly by
\begin{equation}\label{dconst}
C_d:= \frac{\Gamma (1/d)}{d}\, e^{\frac{\pi i}{2d}},
\end{equation}
and where the  remainder term satisfies an estimate
$$
|R(\la, y_1)|\le \frac{C'_d}{(\la y_1^{d(1+a)})^{1/d+\de_1}},
$$
where $\de_1>0$ is a positive number and where the constant $C'_d$ can be chosen independently of $a$ and $b$ (we mention this here for later use). The latter estimate implies that
$$
\Big|\int\limits_{\la y_1^{d(1+a)}>M}y_1^a\chi_0\Big(\frac{y_1}{\varepsilon} \Big)R(\la, y_1)\, dy_1\Big|\le  \frac{C_2}{\la^{1/d}}, 
$$
whereas the corresponding integral over  the principal part of $I_{\rm int}(\la, y_1)$ behaves asymptotically like $c\,{\log \la}/{\la^{1/d}},$ as   required. Explicitly, our argument shows that 
\begin{equation}\label{dconst1}
\lim_{\la\to+\infty} \frac{\la^{1/d}}{\log{\la}} J_0(\la)=4\frac{b-a}{1+b} C_d,
\end{equation}
 if $d$ is even.
 \medskip
 
Finally, if $d$ is odd, a very similar reasoning shows that 
\begin{equation}\label{dconst2}
\lim_{\la\to+\infty} \frac{\la^{1/d}}{\log{\la}} J_0(\la)=2\frac{b-a}{1+b} (C_d+\overline{C_d}).
\end{equation}

\end{proof}

We have thus proved the theorem in the case where  the Newton polyhedron  $\N(\phi)$ has two compact edges containing the vertex $(d,d)$ as one of their endpoints.

\medskip
Assume therefore next that at least one of the two  edges containing the vertex $(d,d)$ is unbounded. We shall then argue in a similar way as in the previous case, however, by approximating the unbounded faces by  compact line segments which have $(d,d)$ as one of their vertices and which lie on supporting lines to $\N(\phi)$ whose angle with the unbounded face tend to zero.

\subsubsection{Two unbounded edges}\label{twou}
Assume next that both edges containing $(d,d)$  are unbounded, i.e., that  $\N(\phi)=(d,d)+\bR^2_+.$ 
Let us then choose numbers $a,b$  such that  $0<a<1<b,$  where later we shall let $a$ tend to $0$ and $b$ to $\infty.$
We associate to $a$ and $b$ weights 
$$
\ka^a:=\Big(\frac 1{(1+a)d},\frac a{(1+a)d}\Big), \quad \ka^b:=\Big(\frac 1{(1+b)d},\frac b{(1+b)d}\Big).
$$
Then the supporting lines mentioned before will be given by $\ka_1^at_1+\ka_2^a=1$ and $\ka_1^bt_1+\ka_2^b=1,$  respectively, and the identities \eqref{bigger} remain valid.

We can then proceed as in the previous case, reducing to the asymptotic analysis of $J(\la),$ which in return is decomposed into $J_0(\la)$ and $J_\infty(\la),$ given by \eqref{j0} and \eqref{j00}, respectively. We further decompose 
$J_\infty(\la)=J_a(\la)+J_b(\la)$
 as in the proof of Lemma \ref{simple2}.

In place of this lemma, we here  have
\begin{lemma}\label{simple3}
Let $\ve>0.$ Then, if  $\Om$ is chosen sufficiently small, the following estimates
\begin{eqnarray}
&&|J_a(\la)|\le A_d\Big(1+\frac{\log{\lambda} }{1+1/a}\Big) \la^{-1/d}, \label{a2}\\
&&|J_b(\la)|\le A_d\Big(1+\frac{ \log{\lambda} }{1+b}\Big) \la^{-1/d} \label{a3}
\end{eqnarray}
hold true,
with a constant $A_d$ which does not depend on $a$ and $b,$ but only on $d$.
\end{lemma}

\begin{proof} We shall prove the estimate for  $J_a(\la);$ the proof of the corresponding estimate for $J_b(\la)$ is obtained by the same kind of reasoning, essentially just by  interchanging the roles of the variables $x_1,x_2$ in the argument.
 Assuming that $\ve$ is chosen sufficiently small, we may  decompose 
$J_a(\la)=J_{a}^0(\la)+J_{a}^\infty(\la),$
where
$$J_a^0(\la):=\int e^{i\la \phi\x}\chi_0\x \chi^0\left(\frac{\ve x_2}{ |x_1|^{a}}\right)\left(1-\chi^0\left(\frac{x_2}{\varepsilon |x_1|^{a}}\right)\right) \,dx
$$
and 
$$J_a^\infty(\la):=\int e^{i\la \phi\x}\chi_0\x \left(1-\chi^0\left(\frac{\ve x_2}{ |x_1|^{a}}\right)\right) \, dx.
$$
Notice that the integrand of $J_a^0(\la)$ is supported where
$$
\ve |x_1|^a\lesssim |x_2|\lesssim \frac 1\ve  |x_1|^a,
$$
and  the integrand of $J_a^\infty(\la)$ is supported where
$$
\frac 1\ve |x_1|^a\lesssim |x_2|.
$$

Using a dyadic decomposition of $J_a^0(\la)$ by means of the dilations $\de_r$ associated to the weight $\ka^a,$  we can estimate $J_a^0(\la)$ in the same way as we did estimate $J_a(\la)$ in the proof of Lemma \ref{simple2}. Notice to this end that the corresponding integrals $J_k(\la)$ will be performed here over a domain where 
$$\ve ^{1/a}\lesssim |x_1|\lesssim 1\mbox{   and  }\ve \lesssim |x_2|\lesssim  1.$$
And, since now  we have $\phi_{\ka^a}\x=x_1^dx_2^d,$ there is no  root of multiplicity $d$ or greater of $\phi_{\ka^a}$ on this domain, hence we obtain the estimate
$$|J_a^0(\la)|\le C_d \la^{-1/d}.
$$
\medskip

As for $J_a^\infty(\la),$ observe first that there is another smooth, even bump function $\tilde\chi^0$ which is identically $1$ near the origin such that $1-\chi^0\left({\ve x_2}/{ |x_1|^{a}}\right)=\tilde\chi^0\left({ x_1}/{ (\ve^{1/a}|x_2|^{1/a})}\right).$ Moreover, even though this function will depend on $a,$ we may assume that its derivatives are uniformly bounded for $0<a<1.$ We accordingly re-write
$$J_a^\infty(\la):=\int e^{i\la \phi\x}\,\chi_0\x \,\tilde\chi^0\left(\frac{ x_1}{ \ve^{1/a}|x_2|^{1/a}}\right) \, dx.
$$
Decomposing the  integral into the contributions by the four quadrants, we reduce our considerations  to estimating the integral 
$$I(\la):=\int_0^\infty\int_0^\infty e^{i\la \phi\x}\,\chi_0\x \,\tilde\chi^0\left(\frac{ x_1}{ \ve^{1/a}x_2^{1/a}}\right) \, dx_1\,dx_2.
$$
Observe next that the phase function can be written as
$$
\phi(x_1,x_2)=x_1^dx_2^d a(x_1,x_2)+\sum_{\nu=0}^{d-1}(x_1^\nu \varphi_\nu (x_2)+x_2^\nu \psi_\nu (x_1)),
$$
where the functions $\varphi_\nu, \psi_\nu$ are smooth and   flat at the origin and where $a$ is  a smooth function such  that $a(0,0)=1$. This shows that the  change of variables 
$$x_1:= x_2^{1/a} y_1, \quad x_2:=y_2$$
will transform the phase function  $\phi$ into a phase function $\tilde\phi$ of the form
$$
\tilde\phi(y_1,y_2)=y_2^{d(1+1/a)} \Big(y_1^d \,\tilde a(y_1,y_2)+\sum_{\nu=0}^{d-1}y_1^\nu \tilde\varphi_\nu (y_2)\Big)=:y_2^{d(1+1/a)}\psi(y_1,y_2),
$$
where the functions $ \tilde\vp_\nu$ are again smooth and flat at the origin and where $\tilde a$ is smooth  with $\tilde a(0,0)=1.$

Accordingly, we re-write
$$I(\la)=\int_0^\infty\int_0^\infty e^{i\la y_2^{d(1+1/a)}\psi(y_1,y_2)}\,\tilde\chi_0(y_1,y_2) \,\tilde\chi^0\left(\frac{ y_1}{ \ve^{1/a}}\right)\,\, dy_1 \,y_2^{1/a} dy_2.
$$
Notice that if $M$ is any fixed positive constant, than the contribution to $I(\la)$ by the region where   $\la y_2^{d(1+1/a)}\le M$ is  trivially bounded by $C_M \lambda^{-1/d},$ with a constant $C_M$ which does not depend on $a,$ so that we may assume that 
$\la y_2^{d(1+1/a)}> M$ in the inner integral, where $M$ is a sufficiently large constant.

In order to estimate the inner integral, observe that the $C^M$-norm of $\psi$ as a function of $y_1$ and $y_2$ may be very large as $a\to 0,$ due  the type of change of coordinates that we applied. However, for $y_2$ fixed, the $d$'th derivative of $\psi$ with respect to $y_1$ is bounded from  below by a fixed constant not depending on $y_2$ and $a.$ Indeed, by choosing $\Om$ sufficiently small, it is easy to see that we may assume that $\pa_1^d\psi(y_1,y_2)\ge a(0,0)d! /2=d! /2.$

 We may thus apply  van der Corput's estimate  in order to estimate the inner integral with respect to $y_1$ by $C (\la y_2^{d(1+1/a)})^{-1/d},$ with a constant $C$ which does not depend on $a,$ 
 and then perform the integration in $y_2,$ to find that
 $$
|I(\la)|\le A_d\Big(1+\frac{\log{\lambda} }{1+1/a}\Big) \la^{-1/d},
$$
as required.
\end{proof}

We are thus left with the main term $J_0(\la),$  which can be treated exactly as in the proof of Lemma \ref{main}, so that the conclusion of this lemma holds true. In particular, the limit relations \eqref{dconst1} and \eqref{dconst2} hold true. Letting  $a\to 0$ and $b\to \infty,$ we finally derive from those limit relations in combination with Lemma \ref{simple3} that indeed 
$$
\lim_{\la\to +\infty}\frac{\la^{1/d}}{\log{\la}}J(\la)=c,
$$
where $c$ is given by $4C_d,$ if $d$ is even, and by $2(C_d+\overline{C_d}),$ if $d$ is odd. This proves Theorem \ref{limit} also in this case.

\subsubsection{A compact and an unbounded edge}\label{twocu}
Finally, if one of the edges containing $(d,d)$ is compact  and the other one is  unbounded, then let us assume without loss of generality that the edge lying above the bi-sectrix is compact, and the one below is unbounded. Then we define $a:=\ka_2^a/\ka_1^a$ associated to the upper, compact edge as in Subsection \ref{twoc}, and approximate the lower, horizontal edge by a compact line segment of slope $1/b$ as in Subsection \ref{twou}, and consider what will happen to the integrals $J_a(\la),$ $J_b(\la)$ and $J_0(\la),$  defined in the same way as before, when $b\to +\infty.$ 

Applying the same kind of reasoning as before, one then finds that $J_a(\la)=O(\la^{-1/d})$ as $ \la\to+\infty,$ that $J_b(\la)$ satisfies  estimate \eqref{a3} from Lemma \ref{simple3}, and that the main contribution is again given by $J_0(\la),$ which can be treated as before by Lemma 3.6. We can then conclude as in the previous case by letting $b\to +\infty.$ 

This completes the proof of Theorem \ref{limit}.
\end{proof}

 \setcounter{equation}{0}
\section{ Fourier restriction in the case of adapted coordinates.}\label{restriction}

Let us  turn to proving the restriction estimate \eqref{rest1} in  Theorem \ref{restrict}.  We may assume that   $x^0=(0,0),$ and that  the hypersurface $S$  is given as the  graph $x_3=\phi\x$ of a smooth, finite type  function $\phi$  in adapted coordinates $\x,$ which is defined in an open neighborhood $\Omega$ of the origin such that $\phi(0,0)=0,\nabla\phi(0,0)=0.$ Recall that then  $\nu(x^0,S)=\nu(\phi)$ and $h(x^0,S)=h=d(\phi).$
\medskip

If $\nu(\phi)=0,$ then by A. Greenleaf's work \cite {greenleaf} (compare also \cite{stein-book}, Ch. VIII, 5.15 (b)), the  $L^p(\RR^3)$-$L^2(S)$ restriction theorem for the Fourier transform in Theorem \ref{restrict} is an immediate consequence of the  uniform estimate in Corollary \ref{s1.3}  for the Fourier transform of the surface carried   measure $\rho d\sigma$ of the hypersurface $S.$

\medskip
We shall therefore assume in the sequel that  $\nu(\phi)=1$. This implies in particular that $h=h(\phi)\ge 2.$  Note that in this case Greenleaf's theorem misses the endpoint $p=p_c=(2h+2)/(2h+1),$ on which we shall concentrate in the sequel. As we shall see, this endpoint can nevertheless be obtained if we invoke tools from Littlewood-Paley theory. Our approach has some resemblance to Stein's proof in \cite{stein-book}, Ch. VIII, 5.16, of  Strichartz' estimates for the Fourier restriction to quadratic surfaces from \cite{strichartz}.
\smallskip

We shall denote by $\mu$ the surface carried  measure $\rho d\si$ from Theorem \ref{restrict}.  By decomposing $\RR^2$ again into its four quadrants, we may assume without loss of generality that $\mu$ is of the form
$$
\laa \mu, f\ra=\int_{(\RR+)^2} f(x',\phi(x'))\, \eta(x')\, dx', \qquad f\in C_0(\RR^3),
$$
where $\eta(x'):=\rho(x',\phi(x'))\sqrt{1+|\nabla \phi(x')|^2}$ is smooth and has its support in a sufficiently small neighborhood $\Omega$ of the origin.
\medskip
In the sequel, we shall split the coordinates in $\RR^3$ as  $x=(x',x_3)\in \RR^2\times \RR.$  If $\chi$ is an integrable funtion defined on $\Omega,$ we put
$$
\mu^\chi:=(\chi\otimes 1)\mu,\  \mbox{ i.e.,}\   \laa \mu^\chi, f\ra=\int_{(\RR+)^2} f(x',\phi(x'))\, \eta(x')\chi(x')\, dx'.
$$
Observe that then
\begin{equation}\label{3.1}
\widehat{\mu^\chi}(-\xi)=J^\chi(\xi), \quad \xi\in\RR^3,
\end{equation}
with $J^\chi(\xi)$ defined as in Section \ref{uniest}.
\bigskip

We next choose  a weight $\ka$ with $0<\ka_1\le \ka_2$ such that the line $L_\ka$ is a supporting line to the Newton polyhedron $\N(\phi)$ and  so that
$$
\frac1{|\ka|}=d_h(\phi_\ka)=h(\phi_\ka)=h.
$$
 This is possible, since according to Lemma \ref{vertex} the principal face $\pi(\phi)$ of $\N(\phi)$ is either a vertex, or a compact edge such  that $m(\phi_\pr)=d(\phi).$ In the first case, we have $\phi_\ka\x=cx_1^hx_2^h,$ and in the second $\phi_\ka=\phi_\pr,$ so that in both cases
  \begin{equation}\label{3.0}
m(\phi_\ka)=h.
\end{equation}
 The corresponding dilations will be denoted by $\de_r.$
Fixing a suitable smooth cut-off function $\chi\ge 0$ on $\RR^2$ supported in an annulus $D$ such that the functions $\chi_k:=\chi\circ \de_{2^k}$ form a partition of unity, we then decompose the measure $\mu$ as
\begin{eqnarray}\label{3.2}
\mu=\sum_{k\ge k_0}\mu_k,
\end{eqnarray}
where $\mu_k:= \mu^{\chi_k}.$  Let us extend the dilations $\de_r$ to $\RR^3$ by putting
$$
\de^e_r(x',x_3):=(r^{\ka_1} x_1, r^{\ka_2} x_2, r x_3).
$$
We re-scale the measure $\mu_k$ by defining $\mu_{0,(k)}:=2^{-k}\mu_k\circ \dee_{2^{-k}},$ i.e.,
\begin{equation}\label{3.3}
\laa \mu_{0,(k)},f\ra=2^{|\ka|k}\laa \mu_k, f\circ \dee_{2^{k}}\ra=\int_{(\RR+)^2} f(x',\phi^k (x'))\, \eta(\de_{2^{-k}}x')\chi(x')\, dx',
\end{equation}
with $\phi^k(x):=2^k\phi(\de_{2^{-k}}x)=\phi_\ka (x)+ \mbox{error terms}.$
This shows that the measures $\mu_{0,(k)}$ are supported on the smooth hypersurfaces  $S^k$ defined as the graph of $\phi^k$, their total variations are uniformly bounded,
i.e.,
$\sup_k\|\mu_{0,(k)}\|_1<\infty,$
and that they are approaching the surface carried  measure $ \mu_{0,(\infty)}$   on  $S$  defined  by
$$\laa \mu_{0,(\infty)},f\ra:=\int_{(\RR+)^2} f(x',\phi (x'))\, \eta(0)\chi(x')\, dx'
$$
as $k\to\infty. $

 We claim that there is a constant $C$   such that
\begin{equation}\label{3.4}
|\widehat{ \mu_{0,(k)}}(\xi)|\le C (1+|\xi|)^{-1/h} \quad \mbox{for every }\xi\in\RR^3, k\ge k_0.
\end{equation}
\smallskip

Indeed, we may again assume  that
$|\xi_1|+|\xi_2|\le \de |\xi_3|,$ where $0<\de\ll1$ is a sufficiently small constant, since for $|\xi_1|+|\xi_2|> \de |\xi_3|$ the estimate \eqref {3.4} follows by an integration by parts, if $\Om$ is chosen small enough, i.e.,  $k_0$ sufficiently large.

We may thus in particular assume that $ |\xi|\sim |\xi_3|.$
Note that    $\eqref{3.1}$ and  $\eqref{3.3}$ show that
 $$
\widehat{ \mu_{0,(k)}}(-\xi)=2^{|\ka| k}J^{\chi_k}(\dee_{2^{k}}\xi).
 $$
Therefore, in view of \eqref{3.0} Êthe estimate \eqref{estjk}  for $J_k(\xi)=J^{\chi_k}(\xi)$ in Subsection \ref{adaptedc} implies  in our case that
$$
|\widehat{ \mu_{0,(k)}}(\xi)|\le C (1+2^{-k}|\dee_{2^{k}}\xi|)^{-1/h},
$$
which yields \eqref{3.4} if $ |\xi|\sim |\xi_3|.$

According to Theorem 1 in \cite{greenleaf}, the estimates in \eqref{3.4} imply
 the restriction estimates
\begin{eqnarray}\label{3.5}
\left(\int|\hat f(x)|^2\,d \mu_{0,(k)}(x)\right)^{1/2}\le C\|f\|_p, \qquad f\in\S(\RR^3),
\end{eqnarray}
with $p=(2h+2)/(2h+1),$ and the proof in  \cite{greenleaf} reveals that the constant $C$ can be chosen independently of $k.$

Let us re-scale these estimates, by putting
$$
f_{(r)}(x):= r^{|\ka|/2}f(\dee_r x),\quad r>0,
$$
for any function $f$ on $\RR^3.$ Then $\widehat f_{(r)}=r^{-|\ka|/2-1}\widehat{f\circ\dee_{r^{-1}}},$
and \eqref{3.5} implies
$$
\int|\hat f(x)|^2\,d \mu_{k}(x)=\int|\widehat f_{(2^{-k})}(x)|^2 \,d \mu_{0,(k)}(x)\le C^2 2^{(|\ka|/2+1)k}
\|f\circ \dee_{2^k}\|_p^2,
$$
hence
\begin{eqnarray}\label{3.6}
\int|\hat f(x)|^2\,d \mu_{k}(x)\le C^2 \|f \|_p^2,
\end{eqnarray}
with a constant $C$ which does not depend in $k.$

Fix a cut-off function $\tilde\chi\inÊC^\infty_0(\RR^2)$ supported in an annulus centered at the origin
such that $\tilde\chi=1$ on the support of $\chi,$ and define dyadic decomposition operators $\Delta'_k$ by
\begin{eqnarray*}
\widehat{\Delta'_kf}(x):=\tilde\chi(\de_{2^k}x')\, \hat f(x',x_3).
\end{eqnarray*}
 Then $\int|\hat f (x)|^2d\mu_k(x)=\int |\widehat{\Delta'_kf}(x)|^2d\mu_k(x),$ so that \eqref{3.6} yields in fact that
\begin{eqnarray*}
\int|\hat f(x)|^2d\mu_k(x)\le C^2\, \|\widehat{\Delta'_kf}\|_{p}^2,
\end{eqnarray*}
for any $k\ge k_0$.  In combination with Minkowski's inequality, this implies
\begin{eqnarray*}
\left(\int|\hat f(x)|^2d\mu(x)\right)^{1/2}=\left(\sum_{k\ge k_0}\int|\hat f(x)|^2d\mu_k(x)\right)^{1/2}\le \left(\sum_{k\ge k_0}\|\Delta'_kf\|_{p}^2\right)^{1/2}\\
=C\left(\left(\sum_{k\ge k_0}\left(\int|\Delta'_kf(x)|^{p}dx\right)^{2/p}\right)^{p/2}\right)^{1/p}\le
C\left\|\left(\sum_{k\ge k_0}|\Delta'_kf(x)|^2\right)^{1/2}\right\|_{L^{p}(\bR^3)},
\end{eqnarray*}
since $p<2$.

Thus, by Littlewood-Paley Theory \cite{stein-book}, we obtain estimate \eqref{rest1}. This completes the proof of Theorem \ref{restrict}.

\setcounter{equation}{0}
\section{Appendix: Proof of Lemma \ref{vertex}} \label{vertcar}

To prove Lemma \ref{vertex}, we shall apply the techniques and results from \cite{IM-ada}, in particular the reasoning in the proof of Lemma 3.2 of that article.

\medskip
In order to show that (a) implies (b), we may assume without loss of generality that the coordinates $x$ are adapted to $\phi,$ and that the principal face $\pi(\phi)$ is a vertex, say $\pi(\phi)=\{(\ell,\ell)\},$ i.e.,
$$
\phi_\pr\x=cx_1^\ell x_2^\ell.
$$
Assume that $y$ is another adapted coordinate system, say $x= F (y),$ where $ F $ is a local, smooth diffeomorphism at the origin, and write 
$\tilde\phi(y):=\phi( F (y)).$

 Possibly after permuting the coordinates $x_1$ and $x_2,$ we may  choose a weight $\ka=(\ka_1,\ka_2)$ with $0<\ka_1\le \ka_2$ in the following way:
 \medskip

 {\bf Case I. } If  $(\ell,\ell)$ is the right endpoint of a compact edge $\ga$ of the Newton diagram of $\phi,$ then we choose the unique weight $\ka$ so  that  $\ga$ lies on  the line $L_\ka:=\{(t_1, t_2)\in\bR^2:\ka_1t_1+\ka_2 t_2=1\}$ (which is then a supporting line to $\N(\phi)$).

 \smallskip
 {\bf Case II. }  Otherwise, i.e., if $\N(\phi)$ is contained in the half-plane  $t_1\ge \ell,$ then we choose $\ka$ so that  the vertex $(\ell,\ell)$ is the unique point of the Newton polyhedron $\N(\phi)$ contained in the supporting line  $L_\ka.$

Permuting the coordinates $y_1$ and $y_2,$ if necessary, we may  assume without loss of generality that
$(x_1,x_2)=(F_1(y_1,y_2), F_2(y_1,y_2))$ satisfies
$\frac{\pa F_j(0,0)}{\pa y_j}\neq0$ for $j=1,2$. Therefore, we can write the functions   $ F _1,  F _2$ in the form
\begin{equation}\label{}
 F _1(y_1,y_2)=y_1\psi_1(y_1,y_2)+\eta_1(y_2),\quad
 F _2(y_1,y_2)=y_2\psi_2(y_1,y_2)+\eta_2(y_1),
\end{equation}

\noi where $\psi_1,\,\psi_2,\, \eta_1,\, \eta_2$ are smooth functions
satisfying
$$
\psi_1(0,0)\neq0,\quad \psi_2(0,0)\neq0, \quad
\eta_1(0)=\eta_2(0)=0.
$$
We may further assume that $\psi_1(0,0)=\psi_2(0,0)=1.$ Denote by $k_j$ the order of vanishing of $\eta_j$ at $0,$ $j=1,2.$ Then clearly $k_j\ge 1.$

Notice that in Case II, we may and shall assume that $\ka_2/\ka_1>k_2.$

\medskip
We first recall some observation from \cite{IM-ada}.  If $ F _\ka$ denotes the $\ka$-principal part of $ F ,$ then
$$
\tilde \phi(y_1,y_2)=\phi_\ka\circ F _\ka(y_1,y_2) +\ \mbox{terms of higher $\ka$-degree,}
$$
 so that
$$
\tilde \phi_\ka=\phi_\ka\circ F _\ka.
$$
 Moreover, $\phi_\ka\circ F _\ka$ is a $\ka$-homogeneous polynomial, so that its Newton diagram $\N_d(\tilde \phi_\ka)$ is again a compact interval (possibly a single point). In case that this interval intersects the bi-sectrix too, then it contains  the principal face of $\N(\tilde \phi).$

\medskip
 \noi {\bf a) The case where $k_2>\frac{\ka_2}{\ka_1},$   and either $\ka_2>\ka_1,$ or $\ka_1=\ka_2$ and $k_1>1.$ }

In this case, one finds that $ F _\ka(y_1,y_2)=(y_1,y_2)$ (see  \cite{IM-ada}), hence $\tilde\phi_\ka=\phi_\ka,$   so that $\pi(\tilde\phi)=\pi(\phi)$ is a vertex.

 \medskip
 \noi {\bf b) The case where $k_2>\frac{\ka_2}{\ka_1},$   $\ka_1=\ka_2$ and $k_1=1.$ }

 Then $k_2>1, k_1=1,$ so that $ F _\ka(y_1,y_2)=(y_1+ay_2,y_2)$ for some constant $a\in\RR,$ hence $ \tilde\phi_\ka(y_1,y_2)=c(y_1+ay_2)^\ell y_2^\ell.$ From a view at the Newton diagram of this polynomial, we see that $\pi(\tilde\phi)=\pi(\phi)$ is a vertex.

\medskip
 \noi {\bf c) The case where $k_2<\frac{\ka_2}{\ka_1}.$}

 As in the proof of Lemma 3.2 in  \cite{IM-ada}, we  then  introduce a second  weight $\mu:=(1,k_2),$ and choose $d>0$ so that the line $L_\mu:=\{(t_1, t_2)\in\bR^2:t_1+k_2 t_2=d\}$ is the supporting line to the Newton polyhedron $\N(\phi).$ It has been shown in the proof of Lemma 3.2 in  \cite{IM-ada} (Case c)) that the principal  face of $\N(\tilde \phi)$  then lies  on the line $L_\mu.$ Noticing that  the line $L_\mu$ is steeper than the line $L_\ka,$  we see that  Case I cannot arise here, since otherwise we would have $d_y<d_x,$ contradicting our assumption that also the coordinates $y$ are adapted.
 And, in Case II, we see that $(\ell,\ell)$ will be the only point of $\N(\phi)$ contained in $L_\mu,$ so that $\phi_\mu=\phi_\pr.$ This shows that $\tilde\phi_\mu=\phi_\mu\circ F_\mu=\phi_\pr\circ F_\mu.$

Moreover, the $\mu$-principal part   of $F$ is given by $F_\mu(y_1,y_2)=(y_1,y_2+a_2y_1^{k_2}),$ if $k_2>1,$ and by $F_\mu(y_1,y_2)=(y_1+a_1y_2,y_2+a_2y_1),$
 if $k_2=1,$ with $a_1\ne 0$ if and only if   $k_1=1.$

 In the first case, we obtain $\tilde\phi_\mu=c y_1^\ell (y_2+a_2y_1^{k_2})^\ell,$ so that $\pi(\tilde\phi)=\pi(\phi)$ is again a  vertex. A similar reasoning applies in the second case, if $a_1=0$ or $a_2=0.$ And, if $a_1\ne 0\ne a_2,$  we find that $\tilde\phi_\mu=c(y_1+a_1y_2)^\ell(y_2+a_2y_1)^\ell.$ This means that the principal face of $\N(\tilde\phi)$ is a compact edge passing through the point $(\ell,\ell),$ and clearly $m(\tilde\phi_\pr)=\ell,$ so
 that $m(\tilde\phi_\pr)=\ell=d(\tilde\phi).$

\medskip
 \noi {\bf e) The case where $k_2=\frac{\ka_2}{\ka_1}.$}

 Observe  that $k_1\ka_2>\ka_1,$ unless $\ka_1=\ka_2$ and $k_1=1,$ since $\ka_1/\ka_2\le 1$ (the latter will only arise in Case I).

 Assuming first that $k_1\ka_2>\ka_1,$ we then see that $\phi_\ka(y_1,y_2)=(y_1,y_2+a_2y_1^{k_2}),$ hence $ \tilde\phi_\ka(y_1,y_2)=cy_1^\ell(y_2+a_2y_1^{k_2})^\ell.$ This shows that again  $\pi(\tilde\phi)=\pi(\phi)$ is a vertex.

 Finally, assume that $\ka_1=\ka_2$ and $k_1=1,$ so that also $k_2=1.$ Then $\phi_\ka$ is of the form
 $\phi_\ka(y_1,y_2)=(y_1+a_1y_2,y_2+a_2y_1),$ hence $ \tilde\phi_\ka(y_1,y_2)=c(y_1+a_1y_2)^\ell(y_2+a_2y_1)^\ell.$ As before, this  means that the principal face of $\N(\tilde\phi)$ is a compact edge passing through the point $(\ell,\ell),$ and we have  $m(\tilde\phi_\pr)=\ell=d(\tilde\phi).$

 \medskip
There remains  to show that (b) implies (a). To this end, we may assume without loss of generality  that $y=x,$ i.e., that $x$  is an adapted coordinate system, and  that $\pi(\phi)$ is a compact edge and $m(\phi_\pr)=d(\phi).$  We shall denote the latter by $d.$
 Let us denote by $(A_0,B_0)$ and $(A_1,B_1)$ the two vertices of $\pi(\phi),$  and assume that $A_0< A_1.$ According to \cite{IM-ada}, displays (3.2) and (3.3), we can then write the principal part of $\phi$ as
$$
\phi_\pr(x_1,x_2)=cx_1^{\al}x_2^{\beta}\prod_{l}(x_2-c_lx_1^m)^{n_l},
$$
where the $c_l$'s are the non-trivial distinct complex roots of the polynomial $\phi_\pr(1, x_2)$ and the $n_l$'s  are their multiplicities. Moreover,  there exists an $l_0$ such that $d=n_{l_0}$ and such that $c_{l_0}$ is real.

We then apply the change of coordinates $y_1:=x_1, y_2:=x_2-c_{l_0}x_1^m,$  which preserves the mixed homogeneity of $\phi_\pr$ and  transforms this polynomial into a polynomial of the same form, $cx_1^{\tilde\al}x_2^{\tilde\beta}\prod_{l}(x_2-\tilde c_lx_1^m)^{n_l},$ but now with $\tilde \beta=d.$ The vertices of the corresponding Newton diagram are given by  $(A_0,B_0)$ and $(\tilde A_1,\tilde B_1)$ and lie on the same line as $(A_0,B_0)$ and $(A_1,B_1)$ (see \cite{IM-ada}), where obviously $\tilde B_1=\tilde \beta=d.$ This shows that $(\tilde A_1,\tilde B_1)=(d,d),$ and consequently the  principal face of the Newton polyhedron of $\tilde\phi$ is given by the vertex $(d,d).$

\endproof

\end{document}